\documentclass{siamart251216}
\usepackage{upquote}
\usepackage{graphicx}
\usepackage{algcompatible}
\usepackage{enumitem}
\usepackage{latexsym}
\usepackage{amssymb}
\usepackage{pdflscape}
\usepackage{mathabx} 
\usepackage{pifont}
\usepackage{booktabs}
\usepackage{array}
\usepackage{multirow}
\graphicspath{{figures/}} 
\newcommand{\eqnref}[1]{(\ref{#1})}
\newcommand{\eps}{\epsilon}
\newcommand{\CS}{c_s}

\newcommand{\xmark}{\ding{55}} 
\newsiamthm{example}{Example}
\newsiamthm{remark}{Remark}

\title{Error bounds for the Sherman-Morrison formula and its modification with improved stability}
\author{Behnam Hashemi\thanks{School of Computing and Mathematical Sciences, University of Leicester, Leicester, LE1 7RH, UK. Supported by EPSRC grant UKRI4034. {\tt b.hashemi@le.ac.uk}.}
\and 
Yuji Nakatsukasa\thanks{Mathematical Institute, University of Oxford, Oxford, OX2 6GG, UK. Supported by EPSRC grants EP/Y010086/1 and EP/Y030990/1. {\tt nakatsukasa@maths.ox.ac.uk}.}
}

\begin{document}
\maketitle

\begin{abstract} 
It is known that the Sherman--Morrison (SM) formula is not numerically stable. In recent work, we introduced SMIR, an algorithm that incorporates iterative refinement to enhance the SM backward error. In this paper we take a different route: adapting an algorithm of Govaerts, originally designed for general bordered linear systems, we develop a modified Sherman–Morrison (MSM) method whose built-in self-correction makes it surprisingly resilient. Whereas standard SM requires the solution of two $n\times n$ linear systems, MSM requires three; by contrast, SMIR requires $2+k$ solves, where $k$ is the number of IR steps and can be substantially larger than three, when IR converges slowly. We then derive backward and forward error bounds for both SM and MSM. The SM backward error bound established here is stronger than the one proved in [Hashemi \& Nakatsukasa 2026]: it accounts for every rounding error and holds with no conditions on the size of the capacitance. From these bounds we extract growth factors that are cheap to compute a posteriori and can be used to certify the backward and forward stability of SM and MSM on a given problem. In our experiments, MSM consistently produces backward stable solutions, and is therefore observed to be forward stable as well. A formal proof of the stability --- or instability --- of MSM remains an open problem.
\end{abstract}
\begin{keywords} 
Sherman-Morrison formula; backward stability; forward stability
\end {keywords}

\section{Introduction}
Solving rank-one perturbed linear systems is an important problem across computational science. Applications arise in least-squares problems, electrical networks, sensitivity analysis in linear programming, quasi-Newton methods in unconstrained optimization, and the numerical solution of ordinary and partial differential equations. In quasi-Newton methods, for example, Hessian approximations are updated by low-rank corrections. In differential equations, such systems arise in a variety of situations. For instance, after solving a problem with one boundary condition, one may need to solve a closely related problem with a slightly different boundary condition. In such cases, the new coefficient matrix can often be obtained from the original one by a rank-one perturbation. See~\cite{Fortunato21, Olver13} for further applications to the numerical solution of differential equations involving almost-banded matrices, which arise as low-rank perturbations of banded matrices. 

The celebrated Sherman--Morrison (SM) formula has been used extensively for solving rank-one perturbed linear systems, often allowing one to take advantage of a previously computed solution. Given a nonsingular $n\times n$ real matrix $A$ and $n \times 1$ vectors $u, v$ and $b$, it solves 
\begin{equation}  
\label{eq:maingoal}
(A+uv^T)x = b  
\end{equation}
via  $x = A^{-1}b - A^{-1}u[(1 + v^T\!A^{-1}u)^{-1}v^T\! A^{-1}b]$, where we note the term inside the bracket is a scalar. To make the computational procedure explicit, we present it in Algorithm~\ref{SM:alg}.
\begin{algorithm}[!h]
\caption{SM: Sherman-Morrison formula for solving $(A+uv^T)x = b$.}
\label{SM:alg}
\begin{algorithmic}[1]
\STATE Solve $Ay = b$ for $y$. 
\STATE Solve $Az=u$ for $z$.  
\STATE Compute $\alpha = v^T y$, the capacitance $\beta = 1+v^T z$, and $\theta = \alpha / \beta$. 
\STATE Output $x = y - z \theta$.
\end{algorithmic}
\end{algorithm}

For an overview of various aspects of the SM formula, we refer the reader to the classic paper by Hager~\cite{Hager89}. The present paper focuses on the backward and forward stability properties of the SM formula. Among the relatively few studies addressing the stability of the SM formula is the work of Yip~\cite{Yip86}. Yip considers solving linear systems with large condition numbers by applying a low-rank perturbation so that the resulting system is better conditioned. Yip’s analysis discusses the effect of the capacitance (denominator in Step 3) being close to zero, showing that this can deteriorate the forward error. Yip's concern is primarily with the forward stability of the SM formula. See also~\cite{Hao21}.

A notable open problem posed in Higham’s seminal monograph on numerical stability~\cite[p.~570]{Higham02} asks whether the SM algorithm~\ref{SM:alg} admits a backward error bound proportional to $\kappa_{\infty}(A) = \|A^{-1}\|_{\infty} \|A\|_{\infty}$, the infinity-norm condition number of $A$. In~\cite{SM25}, we established sufficient conditions under which this question can be answered affirmatively. More precisely, the analysis in \cite{SM25} assumes that
\begin{equation}
\label{hypoth1:eq}
 |v^T \widehat z| > 1.1 
\end{equation}
and focuses on the error in the subtraction, i.e., it assumes that the operations in computing $\widehat \alpha, \widehat \beta$, their division and multiplication with $\widehat z_i$ are performed in exact arithmetic. It then proves that \cite[Prop. 3.2]{SM25}
\[
\| r \| \leq \,\eps_M \frac{\check c}{1 - c_1 \eps_M \kappa(A)} \Big( c \| A\| + \|A + u v^T\| \Big)\, \|A^{-1}\| \, \|b\|.
\]

\paragraph{Contributions and organization of the paper}
Our first aim is to address Higham's question in greater generality: in our analysis in Section 2, we do not assume~\eqnref{hypoth1:eq}, and derive a rigorous error bound for the backward error that takes {\em all} rounding errors into account. Since our new bound does not require~\eqref{hypoth1:eq} to hold, it also clarifies why a small capacitance does not necessarily compromise backward stability of SM; see the discussion above on the effect of capacitance on the forward error, as well as the two illustrative examples in Subsection~\ref{smallBeta:sec}. In addition to a backward error bound, in Section 2, we also derive forward error bounds for the SM solution, thereby also answering a research problem in Higham's book~\cite[p. 487, Problem 26.2 (b)]{Higham02}. In Section 3, we propose a modified Sherman-Morrison (MSM) algorithm (Alg.~\ref{MSM:alg}) whose stability properties are much stronger than the standard SM, namely Alg.~\ref{SM:alg}. Backward and forward error bounds for MSM are established in Section~4 where we also introduce growth factors, that, when of order one, guarantee backward or forward stability of both SM and MSM. To our knowledge, all these bounds are new. In Section~5, we illustrate the error bounds and the performance of MSM through numerical experiments. 

\paragraph{Notation and basics assumptions} The remainder of this section gives a brief summary of the preliminaries used throughout the paper. The standard notation for unit roundoff is $u$. Here, however, we reserve $u$ and $v$ for the vectors forming the rank-one perturbation $uv^T$. We therefore use $\eps_M$ to denote the unit roundoff. We assume that \(n\eps_M<1\) so that  $\gamma_n := \frac{n\eps_M}{1-n\eps_M}=   n\eps_M+\mathcal O(\eps_M^2)$. Throughout the paper, we assume that linear systems with a nonsingular $n \times n$ matrix $A$ can be solved efficiently, either because of the structure of $A$ or because a decomposition of $A$ is already available, such as an LU or QR factorization.\ These are precisely the settings in which the SM formula is a reasonable computational tool~\cite{Hager89}. We also make the standard assumptions that linear systems with $A$ can be solved in a normwise backward stable manner which means that if a solver $S$ is applied for solving $A y = b$ in floating-point arithmetic, then there exists a modest constant\footnote{In practice, $\CS$ is typically a small constant; see, e.g.,~\cite[p. 102]{Golub13} and~\cite[Thm. 16.2 ]{TB:2022}, although rigorous error bounds~\cite[p. 165]{Higham02} replace $\CS \eps_M$ with $n^2 \gamma_{3n} \rho_n$, where $\rho_n$ is the growth factor for Gaussian elimination with partial pivoting.} $\CS$, a matrix $\Delta A$ and a vector $\Delta b$ such that 
\begin{equation}
\label{eq:nrmBackStab_Ayb}
(A + \Delta_1 A) \widehat y = b + \Delta_1 b \quad \mbox{ with } \quad \| \Delta_1 A\| \leq \CS \eps_M \|A\|,\quad \| \Delta_1 b\| \leq \CS \eps_M \|b\|.
\end{equation}
Throughout, $\|\cdot\|$ denotes the Euclidean vector norm and the induced matrix
$2$-norm. We use the standard relative error counter notation~\cite[p. 63]{Higham02}
\[
    \langle k\rangle
    =
    \prod_{i=1}^k(1+\delta_i)^{\rho_i},
    \qquad
    |\delta_i|\leq \eps_M,
    \qquad
    \rho_i = \pm 1
\]
and apply rules such as $ \langle j\rangle\langle k\rangle = \frac{\langle j\rangle}{\langle k\rangle}=\langle j+k\rangle$; see \cite[p. 68]{Higham02}.

To represent the backward error of a basic floating point operation $j = 1,2, \dots$ involving vectors (e.g., scalar-vector multiplication, addition of vectors, inner products), we use diagonal matrices $\Theta_j = I_n + D_j$ where $I_n$ is the $n\times n$ identity matrix and $D_j$ satisfies $|D_j| \leq \eps_M I_n$ or $|D_j| \leq \gamma_n I_n$ depending on the floating point operation in use. See e.g., \cite[p. 198]{Geijn14} and \cite[pp. 63-68]{Higham02}). In both cases, \(\Theta_j\) is nonsingular by Neumann series which also gives
$\Theta_j ^{-1}  = (I+D_j)^{-1} = I - D_j + D_j^2 -\dots$ and
\begin{equation}
\label{eq:Theta-inv-bnd}
\|\Theta_j^{-1}  \| \le \frac{1}{1-\|D_j\|}.
\end{equation}
Specifically, in the case in which $|D_j| \leq \eps_M I_n$, we have
\begin{equation}
\label{eq:Theta-I-bnd}
\|\Theta_j  - I\| \le \eps_M, \qquad \mbox{ while } \qquad \|\Theta_j^{-1}  - I\|  \le \gamma_1 = \frac{\eps_M}{1-\eps_M}.
\end{equation}

\section{Error bounds for SM} 
We begin with a backward error representation for the SM solution.
\begin{theorem} \label{thm:SM-back-bnd}
Assume linear systems $Ay=b$ and $Az=u$ are solved normwise backward stably. Then, the computed solution $\widehat x$ of SM in floating point arithmetic satisfies
\begin{equation} \label{eq:SM-back-bnd}
\Big( (A+\Delta A) + (u+\Delta u) (v+\Delta v)^T \Big) \widehat x = b + \Delta b + \Delta F\ \widehat y.
\end{equation}
where
\begin{align}
&     \|\Delta A\| \leq \Big(   (2+\CS)\eps_M + \mathcal O(\eps_M^2) \Big) \|A\| \label{eq:DeltABnd}\\
& \|\Delta u\| \le \CS \eps_M \|u\| \nonumber\\
& \|\Delta v\| \le 
   \Big( (n+2)\eps_M +\mathcal{O}(\eps_M^2) \Big) \|v\| \label{eq:DeltvBnd}\\
& \|\Delta b\| \le \CS \eps_M \|b\| \nonumber\\   
&    \|\Delta F\|  \leq \eps_M \Big(
    (2\CS+1)\|A\|
    +
    (2n+3)\|u\|\,\|v\|  \Big)
    +
    \mathcal O(\eps_M^2).\label{eq:T_Bnd}
\end{align}
and $\CS$ is the stability constant of the solver with $A$ as in~\eqref{eq:nrmBackStab_Ayb}. 
\end{theorem}
Let us note that~\eqref{eq:SM-back-bnd} almost implies SM is backward stable; it would if the final term $\Delta F\widehat y$ was not present.
\begin{proof} 
We begin from the floating-point inner products in Step 3 of Alg.~\ref{SM:alg}. The following basic relations hold (see e.g., \cite[p. 198]{Geijn14} and \cite[pp. 63-68]{Higham02}):
\[
    \operatorname{fl}(v^T\widehat y)
    =
    v^T\Theta_1 \widehat y,
    \qquad
    \operatorname{fl}(v^T\widehat z)
    =
    v^T\Theta_2 \widehat z,
\]
where $\Theta_1  = I+D_y$ and $\Theta_2  = I_n+D_z$ with
\begin{align*}
&    D_y
    =
    \operatorname{diag}
    \bigl(
        \theta_n^{(y,1)},
        \theta_n^{(y,2)},
        \theta_{n-1}^{(y,3)},
        \ldots,
        \theta_2^{(y,n)}
    \bigr),\\
    &    D_z
    =
    \operatorname{diag}
    \bigl(
        \theta_n^{(z,1)},
        \theta_n^{(z,2)},
        \theta_{n-1}^{(z,3)},
        \ldots,
        \theta_2^{(z,n)}
    \bigr).
\end{align*}
Here
$  |\theta_j^{(y,i)}|\le \gamma_j$ and $  |\theta_j^{(z,i)}|\le \gamma_j.$
In particular,
\begin{equation}
\label{eq:D_yzBnd}
    |D_y|\le \gamma_n I,
    \qquad
    |D_z|\le \gamma_n I.
\end{equation}
In SM we compute $\widehat \alpha = fl(v^T \widehat y) =v^T\Theta_1 \widehat y$.
Similarly
\[
\widehat \beta = fl(1 + v^T \widehat z) = 1 \oplus v^T\Theta_2 \widehat z = (1 + v^T\Theta_2 \widehat z) \langle 1\rangle
\]
Then, 
\begin{equation}
\label{eq:theta_alpha_beta}
\widehat \theta := fl( \frac{\widehat \alpha}{\widehat \beta}) =  \frac{\widehat \alpha}{\widehat \beta}  \langle 1\rangle = \frac{\widehat \alpha}{1 + v^T\Theta_2 \widehat z}  \frac{\langle 1\rangle}{\langle 1\rangle}  
= \frac{v^T\Theta_1 \widehat y}{1 + v^T\Theta_2 \widehat z}  \langle 2\rangle.
\end{equation}

In the final step of SM, we have
\[
\widehat x = \widehat y \ominus (\widehat z \otimes \widehat \theta)
\]
For each entry $\widehat x_i := \widehat y_i \ominus (\widehat z_i \otimes \widehat \theta)$ we have
\[
\widehat z_i \otimes \widehat \theta = (\widehat z_i  \widehat \theta) (1+\delta_i^{(z)}), \qquad \mbox{ where } \qquad |\delta_i^{(z)}| \leq \eps_M
\]
and
\[
\widehat y_i \ominus \widehat z_i \otimes \widehat \theta = (\widehat y_i - \widehat z_i  \otimes \widehat \theta) (1+\delta_i^{(y)}), \qquad \mbox{ where } \qquad |\delta_i^{(y)}| \leq \eps_M.
\]
Hence, there are $n \times n$ diagonal matrices $\Theta_3$ and $\Theta_4$ such that 
\[
\widehat z \otimes \widehat \theta = \Theta_3 \widehat z  \widehat \theta, \qquad \mbox{ and } \qquad
\widehat x = \widehat y \ominus  \Theta_3 \widehat z  \widehat \theta = \Theta_4 \Big(\widehat y - \Theta_3 \widehat z  \widehat \theta \Big).
\]
As a result
\begin{equation}
\label{eq:z_theta}
\widehat z  \widehat \theta = \Theta_3^{-1} \widehat y - \Theta_3^{-1} \Theta_4^{-1} \widehat x
\end{equation}

Next, from the backward stability of solver for $Az=u$ we have
\begin{equation}
\label{eq:nrmBackStab_Azu}
(A+\Delta_2 A) \widehat z  = u + \Delta u, \quad \|\Delta_2 A\| \leq \CS \eps_M \|A\|,
\qquad
\|\Delta u\| \le \CS \eps_M \|u\|.
\end{equation}
Multiplying by $\widehat \theta$ yields
\[
(A+\Delta_2 A) \widehat z  \widehat \theta = (u + \Delta u) \widehat \theta,
\]
which by~\eqref{eq:z_theta} gives
\[
(A+\Delta_2 A) \Big(\Theta_3^{-1} \widehat y - \Theta_3^{-1} \Theta_4^{-1} \widehat x \Big)= (u + \Delta u) \widehat \theta,
\]
Hence
\begin{equation}
\label{eq:tempor1}
(A+\Delta_2 A) \Theta_3^{-1} \Theta_4^{-1} \widehat x  + (u + \Delta u) \widehat \theta = (A+\Delta_2 A) \Theta_3^{-1} \widehat y ,
\end{equation}

On the other hand, a bit of algebraic manipulations gives
\[
0 = (A + \Delta_1 A) \widehat y + (\Delta_2 A - \Delta_1 A) \widehat y - (A+ \Delta_2 A) \widehat y
\]
Adding the right-hand side of~\eqref{eq:tempor1} to both sides gives
\[
(A+\Delta_2 A) \Theta_3^{-1} \widehat y = (A + \Delta_1 A) \widehat y  + (\Delta_2 A - \Delta_1 A) \widehat y + (A+\Delta_2 A) \Theta_3^{-1} \widehat y  - (A+ \Delta_2 A) \widehat y
\]
i.e.,
\begin{equation}
\label{eq:tempor2}
(A+\Delta_2 A) \Theta_3^{-1} \widehat y = (A + \Delta_1 A) \widehat y  + (\Delta_2 A - \Delta_1 A) \widehat y + (A+\Delta_2 A) \Big( \Theta_3^{-1} - I \Big) \widehat y
\end{equation}
Now, from~\eqref{eq:nrmBackStab_Ayb}, we substitute $(A+\Delta_1 A) \widehat y  = b + \Delta_1 b$ in~\eqref{eq:tempor2} to get 
\[
(A+\Delta_2 A) \Theta_3^{-1} \widehat y = b + \Delta_1 b + \Big( (\Delta_2 A - \Delta_1 A) + (A+\Delta_2 A) ( \Theta_3^{-1} - I) \Big) \widehat y
\]
Together with~\eqref{eq:tempor1}, this yields
\begin{align}
(A+\Delta_2 A) \Theta_3^{-1} \Theta_4^{-1} \widehat x  + (u + \Delta u) \widehat \theta & = b + \Delta_1 b  \nonumber \\
& \hspace*{-1cm} + \Big( (\Delta_2 A - \Delta_1 A) + (A+\Delta_2 A) ( \Theta_3^{-1} - I) \Big) \widehat y. \label{eq:tempor3}
\end{align}
Next, we focus on the computed value of $\theta$. According to~\eqref{eq:theta_alpha_beta}, \(\widehat \theta\) satisfies
\[
\widehat \theta +  v^T\Theta_2 \widehat z \widehat \theta 
= v^T\Theta_1 \widehat y  \langle 2\rangle. 
\]
Using~\eqref{eq:z_theta}, we obtain
\[
\widehat \theta = v^T\Theta_2  \Theta_3^{-1} \Theta_4^{-1} \widehat x + v^T \Big( \Theta_1 \langle 2\rangle - \Theta_2 \Theta_3^{-1} \Big) \widehat y. 
\]
Defining
\begin{align}
   & (v+\Delta v)^T
    := v^T\Theta_2  \Theta_3^{-1} \Theta_4^{-1} =  v^T(I+D_z)\Theta_{3}^{-1}\Theta_4^{-1}, \label{eq:vDeltavDef}\\
&  \widetilde s^T    := v^T \Big( \Theta_1 \langle 2\rangle - \Theta_2 \Theta_3^{-1} \Big)  \nonumber
\end{align}
gives
\[
\widehat \theta = (v+ \Delta v)^T \widehat x +\widetilde s^T \widehat y , 
\]
which together with~\eqref{eq:tempor3} yields
\begin{align*}
\Big( (A+\Delta_2 A) & \Theta_3^{-1} \Theta_4^{-1}  + (u + \Delta u) (v+ \Delta v)^T \Big) \widehat x  = \\
& b + \Delta_1 b + \Big( (\Delta_2 A - \Delta_1 A) + (A+\Delta_2 A) ( \Theta_3^{-1} - I) - (u + \Delta u) \widetilde s^T \Big) \widehat y.
\end{align*}
Defining further 
\begin{align}
& A+\Delta A:= (A+\Delta_2 A) \Theta_3^{-1} \Theta_4^{-1} \label{eq:ADeltaADef}\\
& \hspace*{-0.2cm}  \Delta F := (\Delta_2 A - \Delta_1 A) + (A+\Delta_2 A) ( \Theta_3^{-1} - I) - (u + \Delta u) v^T \Big( \Theta_1 \langle 2\rangle - \Theta_2 \Theta_3^{-1} \Big)  \label{eq:sTDef}
\end{align}
yields~\eqref{eq:SM-back-bnd}. We have already found bounds for the norm of \(\Delta u\) and \(\Delta b\), where we set $\Delta b := \Delta_1 b$. It remains to bound $\|\Delta A\|$, $\| \Delta v\|$ and $\|\Delta F\|$. 

Starting from $\Delta A$, recall that $\Theta_3$ represents the backward error in the vector-scalar multiplication $\widehat z \otimes \widehat \theta$ and $\Theta_4$ represents the backward error in the vector subtraction $\widehat y \ominus (\widehat z \otimes  \widehat \theta)$. Therefore, we have $\Theta_3 = I+D_3$ and $\Theta_4 = I+D_4$ for some $n \times n$ diagonal matrices \(D_3\) and \(D_4\) satisfying
$\|D_3\| \leq \eps_M$ and $\|D_4\| \leq \eps_M$. We therefore have (see~\eqref{eq:Theta-inv-bnd})
\begin{equation}
\label{eq:theta3Bnd}
    \|\Theta_3^{-1}\| \leq \frac{1}{1-\eps_M},
    \qquad
    \|\Theta_4^{-1}\| \leq \frac{1}{1-\eps_M}.
\end{equation}
From~\eqref{eq:ADeltaADef} we see that
\[
    \Delta A
    =
    (A+\Delta_2 A)\Theta_3^{-1}\Theta_4^{-1} - A  =
    A(\Theta_3^{-1}\Theta_4^{-1}-I)
    +
    \Delta_2 A\,\Theta_3^{-1}\Theta_4^{-1}.
\]
Therefore,
\[
\begin{aligned}
    \|\Delta A\|
    &\leq
    \|A\|\,\|\Theta_3^{-1}\Theta_4^{-1}-I\|
    +
    \|\Delta_2 A\|\,\|\Theta_3^{-1}\|\,\|\Theta_4^{-1}\| .
\end{aligned}
\]
Moreover,
\[
    \|\Theta_3^{-1}\Theta_4^{-1}-I\|
    =
    \|\Theta_3^{-1}\Theta_4^{-1}
      -\Theta_3^{-1}
      +\Theta_3^{-1}-I\|  
    \leq
    \|\Theta_3^{-1}\|\,\|\Theta_4^{-1}-I\|
    +
    \|\Theta_3^{-1}-I\|.
\]
By~\eqref{eq:Theta-I-bnd} 
\begin{equation}
\label{eq:thetaInvs}
    \|\Theta_j^{-1}-I\|
    \leq
    \frac{\eps_M}{1-\eps_M} = \gamma_1,
    \qquad j=3,4,
\end{equation}
\begin{equation}
\label{eq:theta34nrm}
    \|\Theta_3^{-1}\Theta_4^{-1}-I\|
    \leq
    \frac{1}{1-\eps_M}\frac{\eps_M}{1-\eps_M}
    +
    \frac{\eps_M}{1-\eps_M}        
    =
    \frac{2\eps_M-\eps_M^2}{(1-\eps_M)^2}.
\end{equation}
Using also $\|\Delta_2 A\| \leq \CS \eps_M \|A\|$, we therefore have
\[
    \|\Delta A\| \leq
    \|A\|
    \frac{2\eps_M-\eps_M^2}{(1-\eps_M)^2}
    +
    \CS \eps_M \|A\|
    \frac{1}{(1-\eps_M)^2}  
    =
    \frac{(2+\CS)\eps_M-\eps_M^2}{(1-\eps_M)^2}\,\|A\|
\]
which gives~\eqref{eq:DeltABnd}.

We next show~\eqref{eq:DeltvBnd}. From~\eqref{eq:vDeltavDef} we have 
$  \Delta v^T  =    v^T\left((I+D_z)\Theta_3^{-1}\Theta_4^{-1}-I\right)$ and hence
\[
    \|\Delta v\|
    \leq
    \|v\|
    \left\|
    (I+D_z)\Theta_3^{-1}\Theta_4^{-1}-I
    \right\|.
\]
Now
\[
\begin{aligned}
    (I+D_z)\Theta_3^{-1}\Theta_4^{-1}-I
    &=
    \Theta_3^{-1}\Theta_4^{-1}-I
    +
    D_z\Theta_3^{-1}\Theta_4^{-1}.
\end{aligned}
\]
Therefore,
\[
\begin{aligned}
    \left\|
    (I+D_z)\Theta_3^{-1}\Theta_4^{-1}-I
    \right\|
    &\leq
    \|\Theta_3^{-1}\Theta_4^{-1}-I\|
    +
    \|D_z\|\,\|\Theta_3^{-1}\|\,\|\Theta_4^{-1}\|.
\end{aligned}
\]
Using previous bounds~\eqref{eq:theta34nrm},~\eqref{eq:theta3Bnd} and~\eqref{eq:D_yzBnd}, we get
\[
    \left\|
    (I+D_z)\Theta_3^{-1}\Theta_4^{-1}-I
    \right\|
    \leq
    \frac{2\eps_M-\eps_M^2}{(1-\eps_M)^2}
    +
    \frac{\gamma_n}{(1-\eps_M)^2}        =
    \frac{\gamma_n+2\eps_M-\eps_M^2}{(1-\eps_M)^2}.
\]
From
\[
    \gamma_n + 2\eps_M-\eps_M^2
    =
   n\eps_M+\mathcal O(\eps_M^2) + 2\eps_M+\mathcal O(\eps_M^2)
    =
    (n+2)\eps_M+\mathcal O(\eps_M^2)
\]
and $(1-\eps_M)^{-2}    =    1+2\eps_M+\mathcal O(\eps_M^2)$
we get
\[
    \frac{\gamma_n+2\eps_M-\eps_M^2}{(1-\eps_M)^2}
    =
    \Big((n+2)\eps_M+\mathcal O(\eps_M^2)\Big)
    \Big(1+2\eps_M+\mathcal O(\eps_M^2)\Big)  
    =
    (n+2)\eps_M+\mathcal O(\eps_M^2)
\]
which gives~\eqref{eq:DeltvBnd}.

Finally, we prove~\eqref{eq:T_Bnd}. Taking norms in~\eqref{eq:sTDef} gives
\[
\begin{aligned}
\hspace*{-0.2cm}
    \|\Delta F \|
    &\hspace*{-0.1cm} \leq
    \|\Delta_2 A-\Delta_1 A\|
    +
    \|(A+\Delta_2 A)(\Theta_3^{-1}-I)\|      
    +
    \|(u+\Delta u)v^T
    [
        \Theta_1\langle 2\rangle
        -
        \Theta_2\Theta_3^{-1}
    ]\| \\
    &\hspace*{-0.1cm} \leq
    \|\Delta_2 A\|
    +
    \|\Delta_1 A\|
    +
    \bigl(\|A\|+\|\Delta_2 A\|\bigr)
    \|\Theta_3^{-1}-I\|                    
    +
    \|u+\Delta u\|\,\|v\|\,
    \|\Theta_1\langle 2\rangle
    -
    \Theta_2\Theta_3^{-1}\|.
\end{aligned}
\]
In view of the backward stability of the solver with $A$, i.e.,~\eqref{eq:nrmBackStab_Ayb} and~\eqref{eq:nrmBackStab_Azu}, and using~\eqref{eq:thetaInvs} for $j=3$, we obtain
\[
\begin{aligned}
    \|\Delta F \|
    &\leq
    2\CS\eps_M\|A\|
    +
    (1+\CS\eps_M)\|A\|
    \frac{\eps_M}{1-\eps_M}                 \\
    &\hspace{2em}
    +
    (1+\CS \eps_M)\|u\|\,\|v\|\,
    \|\Theta_1\langle 2\rangle
    -
    \Theta_2\Theta_3^{-1}\|.
\end{aligned}
\]
It remains to bound the last term. We have
\[
    \|\Theta_1\langle 2\rangle   -    \Theta_2\Theta_3^{-1}\|
    = \|  \bigl(\Theta_1\langle 2\rangle-I\bigr)
    -
    \bigl(\Theta_2\Theta_3^{-1}-I\bigr) \| \leq
    \|\Theta_1\langle 2\rangle-I\|
    +
    \|\Theta_2\Theta_3^{-1}-I\|.
\]
From
\[
    \langle 2\rangle=(1+\delta_1)(1+\delta_2),
    \qquad
    |\delta_i|\leq \eps_M,
    \qquad i=1,2,
\]
we see that
$    |\langle 2\rangle|
    \leq
    (1+\eps_M)^2$, 
and $
    |\langle 2\rangle-1|
    =
    |\delta_1+\delta_2+\delta_1\delta_2|
    \leq
    2\eps_M+\eps_M^2$. Therefore, 
    \[
\begin{aligned}
    \|\Theta_1\langle 2\rangle-I\|
    &=
    \|(I+D_y)\langle 2\rangle-I\|   =
    \|D_y\langle 2\rangle
      +(\langle 2\rangle-1)I\|                 \\
    &\leq
    |\langle 2\rangle|\,\|D_y\|
    +
    |\langle 2\rangle-1|  \leq
    (1+\eps_M)^2\gamma_n
    +
    2\eps_M+\eps_M^2 .
\end{aligned}
\]
For the second term,
\[
\begin{aligned}
    \|\Theta_2\Theta_3^{-1}-I\|
    &=
    \|(I+D_z)\Theta_3^{-1}-I\|  =
    \|(\Theta_3^{-1}-I)+D_z\Theta_3^{-1}\|      \\
    &\leq
    \|\Theta_3^{-1}-I\|
    +
    \|D_z\|\,\|\Theta_3^{-1}\|                  \leq
    \frac{\eps_M}{1-\eps_M}
    +
    \frac{\gamma_n}{1-\eps_M} =
    \frac{\gamma_n+\eps_M}{1-\eps_M}.
\end{aligned}
\]
Hence
\[
\begin{aligned}
    \|\Theta_1\langle 2\rangle
    -
    \Theta_2\Theta_3^{-1}\|
    &\leq
    (1+\eps_M)^2\gamma_n
    +
    2\eps_M+\eps_M^2
    +
    \frac{\gamma_n+\eps_M}{1-\eps_M}.
\end{aligned}
\]
Substituting this into the previous estimate gives
\[
\begin{aligned}
    \|\Delta F \|
    &\leq
    \left(
        2\CS\eps_M
        +
        \frac{\eps_M}{1-\eps_M}
        +
        \frac{\CS\eps_M^2}{1-\eps_M}
    \right)\|A\|                                      \\
    &\hspace{2em}
    +
    (1+\CS \eps_M)
    \left(
        (1+\eps_M)^2\gamma_n
        +
        2\eps_M+\eps_M^2
        +
        \frac{\gamma_n+\eps_M}{1-\eps_M}
    \right)
    \|u\|\,\|v\|.
\end{aligned}
\]
Using $\frac{\eps_M}{1-\eps_M}
    =
    \eps_M+\mathcal O(\eps_M^2)$, again, $\gamma_n=n\eps_M+\mathcal O(\eps_M^2)$ and
\[
    (1+\eps_M)^2\gamma_n
    +
    2\eps_M+\eps_M^2
    +
    \frac{\gamma_n+\eps_M}{1-\eps_M}
    =
    2\gamma_n+3\eps_M
    +
    \mathcal O(\eps_M\gamma_n+\eps_M^2) = 
    (2n +3)\eps_M
    +
    \mathcal O(\eps_M^2)    
\]
yields
\[
    \|\Delta F \|
    \leq
    (2\CS+1)\eps_M\|A\|
    +
    (2n+3)\eps_M\|u\|\,\|v\|
    +
    \mathcal O(\eps_M^2)
    \bigl(\|A\|+\|u\|\,\|v\|\bigr).
\]
\end{proof}

Recall from the Rigal-Gaches theorem~\cite[p.~120]{Higham02} that any approximate solution $\widehat x$ of a linear system $Bx=b$ with nonsingular $B$ satisfies 
\begin{equation}
\label{eq:backErrB_def}
\omega(\widehat{x}) = \frac{\| b - B \widehat x\|}{\|B\|\, \|\widehat{x}\| + \|b\|}.
\end{equation}
We now use the previous result to establish an upper bound for the residual $r$, and the relative backward error of the SM solution.
\begin{theorem}   \label{thm:new-SM-BackErrBnd}
Assume linear systems $Ay=b$ and $Az=u$ are solved normwise backward stably. Then, the computed solution $\widehat x$ of SM in floating point arithmetic satisfies 
\begin{align}
\omega(\hat{x}) & \leq  \ 
 \eps_M (n+2 + \CS) \frac{\big( \| A\| + \|u\| \|v\| \big)\,\| \widehat x\| + \|b\|}{\|A+u v^{T}\|\, \|\widehat x\| + \|b\|} \nonumber \\ 
 & \quad + \eps_M (2n+3+2\CS) \frac{\big( \|A\| + \|u\| \|v\| \big)\,\|\widehat y\|}{\|A+u v^{T}\|\, \|\widehat x\| + \|b\|} + \mathcal{O}(\eps_M^2).
  \label{newBackErrBnd1:eq}
\end{align}
\end{theorem}

\begin{proof}
Following~\eqref{eq:SM-back-bnd}, the SM residual $r := b - (A+uv^T)\hat x$ is
\[
r = \big( \Delta A + u\ \Delta v^T+\Delta u\ v^T+ \Delta u\ \Delta v^T \big) \widehat x - \Delta b - \Delta F\ \widehat y 
\]
Taking norms from both sides yields:
\begin{align*}
\| r \| & \leq \Big( \| \Delta A \| + \|u\|\ \|\Delta v\|+ \|\Delta u\| \|v\| + \|\Delta u\| \ \|\Delta v\| \Big) \|\widehat x\| + \|\Delta b\| +\| \Delta F\|\ \|\widehat y\| \\
& \leq \eps_M \Big(   (2+\CS)  \|A\| +    (n+2) \|u\|\ \|v\| + \CS \|u\| \|v\| \Big) \|\widehat x\| + \CS \eps_M \|b\| \\
& \quad + 
\eps_M \Big(
    (2\CS+1)\|A\|
    +
    (2n+3)\|u\|\,\|v\|  \Big) \ \|\widehat y\| 
+ \mathcal O(\eps_M^2)\\
&\leq \eps_M (n+2+\CS) \Big(  \| A\| + \|u\| \|v\| \Big)\,\| \widehat x\| \, +\, \eps_M\, \CS\, \|b\| \, \nonumber \\[4pt]
& \ \ \ +\, \eps_M  (2n+3+2\CS) \Big( \|A\| + \|u\| \|v\| \Big)\,\|\widehat y\| + \mathcal{O}(\eps_M^2)
 \label{newResBnd1:eq}
\end{align*}
Formula~\eqref{eq:backErrB_def} gives~\eqref{newBackErrBnd1:eq}.
\end{proof}

\begin{corollary} \normalfont
Using the bound 
\[
\| \widehat y\| \leq \frac{\|A^{-1}\| \|b\|}{1 - c_1 \eps_M \kappa(A)}
\]
(see Formula (3.13) in~\cite{SM25}) gives the following bound on the residual 
\begin{align}
\| r \| \leq  &\, \eps_M (n+2+\CS) \Big( \| A\| + \|u\| \|v\| \Big)\,\| \widehat x\| \, +\, \eps_M\, \CS\, \|b\| \, \nonumber \\[4pt]
 & \ \ \ +\, \eps_M \frac{2n+3+2\CS}{1 - c_1 \eps_M \kappa(A)} \Big( \|A\| + \|u\| \|v\| \Big)\,\|A^{-1}\| \, \|b\| + \mathcal{O}(\eps_M^2).
 \label{newResBnd:eq}
\end{align}
Similarly, the relative backward error of the SM solution $\widehat{x}$ satisfies the following bound
\begin{align}
\omega(\widehat{x}) & \leq  \ 
 \eps_M (n+2+\CS) \frac{\Big( \| A\| + \|u\| \|v\| \Big)\,\| \widehat x\| + \|b\|}{\|A+u v^{T}\|\, \|\widehat{x}\| + \|b\|} \nonumber \\ 
 & \ + \eps_M \frac{2n+3+2\CS}{1 - c_1 \eps_M \kappa(A)} \frac{\Big( \|A\| + \|u\| \|v\| \Big)\,\|A^{-1}\| \, \|b\|}{\|A+u v^{T}\|\, \|\widehat{x}\| + \|b\|} + \mathcal{O}(\eps_M^2).
  \label{newBackErrBnd:eq}
\end{align}
While the first term on the right-hand side can typically be expected to be $\mathcal{O}(\eps_M)$, the second term is $\mathcal{O}(\eps_M \kappa(A))$ as Higham conjectured. The bound also suggests that the backward error could be large when $\| A + uv^T\| \ll \| A\| + \|u\| \|v\|$.
\end{corollary}

\subsection{Forward error bounds for the SM solution}
Using Theorem~\ref{thm:SM-back-bnd}, we now establish a forward error bound for the SM solution. It implies that if  \(B:= A + uv^T\) is well-conditioned, then the accuracy of the SM computed solution is determined by the norm of $\widehat y$.
\begin{theorem}
\label{thm:SM_FerrBnd}
Assume linear systems $Ay=b$ and $Az=u$ are solved normwise backward stably. Define
\begin{align*}
& \widetilde c_1 := (n+\CS + 2) \frac{\|A\| + \|u\| \|v\|}{\|B\|} \qquad \widetilde c_2 := (2n+2\CS+3) \frac{\|A\| + \|u\| \|v\|}{\|B\|}\\
& c_2 := \frac{\widetilde c_2 }{1-\eps_M (2\widetilde c_1 + \CS) \kappa(B)}
\end{align*}
and assume that $\eps_M (2\widetilde c_1 + \CS) \kappa(B)<1$. If $\widehat x$ is computed with SM in floating point arithmetic, then
\begin{equation}\label{eq:SM_FerrBnd}
\| \widehat x-x \|
\le c_2 \eps_M \kappa(B) \big( \|\widehat x\| + \|\widehat y\| \big).
\end{equation}
\end{theorem}

\begin{proof}
As the assumptions of Theorem~\ref{thm:SM-back-bnd} are satisfied, we substitute $b = (A+uv^T) x$ into~\eqref{eq:SM-back-bnd} to get
\[
\Big( (A+\Delta A) + (u+\Delta u) (v+\Delta v)^T \Big) \widehat x = ( A + uv^T) x + \Delta b + \Delta F\ \widehat y.
\]
Defining $e:= \widehat x - x$, $B:= A+uv^T$ and $\Delta B:= \Delta A + u\ \Delta v^T+\Delta u\ v^T +\Delta u\ \Delta v^T$ and subtracting $\Delta B\ x$ from both sides yield
\[
(B + \Delta B) e =  \Delta b + \Delta F\ \widehat y - \Delta B\ x 
\]
Multiplying by $B^{-1}$ from the left gives
\[
\big( I + B^{-1}  \Delta B \big) e =   B^{-1} \big( \Delta b + \Delta F\ \widehat y - \Delta B\ x \big)
\]
Now set $E:=B^{-1}\Delta B$. Then
\begin{equation}
\label{eq:IplusEe}
(I+E)e = B^{-1}\bigl(\Delta b + \Delta F\ \widehat y - \Delta B\ x \bigr).
\end{equation}
The definition of $\widetilde c_1$ guarantees that $\|\Delta B\| \le \eps_M \widetilde c_1 \|B\|$; see Lemma~\ref{lem:C_B_prime} in Appendix.
It follows that
\[
\|E\|
\le
\|B^{-1}\| \|\Delta B\|
\le
\eps_M \widetilde c_1 \| B^{-1}\| \| B\|
=
\eps_M \widetilde c_1 \kappa(B) < 1
\]
where the last inequality holds by assumption. Thus \(I+E\) is nonsingular, and the Neumann-series bound gives
\[
\|(I+E)^{-1}\|
\le
\frac{1}{1-\|E\|}
=
\frac{1}{1-\|B^{-1}\Delta B\|}.
\]
Hence, multiplying both sides of~\eqref{eq:IplusEe} by $(I+E)^{-1}$, taking norms, using the bounds on $\|\Delta b\|$ and $\|\Delta F\|$ from Theorem~\ref{thm:SM-back-bnd}, and $b = Bx$ give
\begin{align}
\|e\| &\leq \| (I+E)^{-1}\| \|B^{-1}\| \Big( \| \Delta b\| + \| \Delta F\| \ \| \widehat y\| + \|\Delta B\| \ \|x \| \Big) \nonumber\\
 & \leq \frac{\|B^{-1}\|}{1-\|B^{-1}\Delta B\|}  \Big( \| \Delta b\| + \| \Delta F\| \ \| \widehat y\| + \|\Delta B\| \ \|x \| \Big) \nonumber\\
  & \leq  \frac{\eps_M \|B^{-1}\|}{1-\|B^{-1}\Delta B\|}  \Bigg(  \CS \|b\|  + \Big(
    (2\CS+1)\|A\| +   (2n+3)\|u\|\,\|v\|  \Big)\ \| \widehat y\| + \widetilde c_1 \|B\| \ \|x \| \Bigg) \nonumber\\
  & \leq  \frac{\eps_M \|B^{-1}\|}{1-\|B^{-1}\Delta B\|}  \Bigg(  
  \CS \|B x\|  +   (2n+3+2\CS) \Big( (\|A\| + \|u\|\,\|v\|)  \Big)\ \| \widehat y\| + \widetilde c_1 \|B\| \ \|x \| \Bigg) \label{ferr_SM_proofStep:eq}\\    
  & =  \frac{\eps_M \|B^{-1}\|}{1-\|B^{-1}\Delta B\|}  \Big(  \CS \|B\| \|x\|  + \widetilde c_1 \|B\| \|x \| +
\widetilde c_2  \|B\| \| \widehat y\|  \Big) \nonumber
\end{align}
where the last equality (up to second-order terms in $\eps_M$) uses the definition of $\widetilde c_2$ in the theorem statement. Since $\|B^{-1} \Delta B\| \le \eps_M \widetilde c_1 \kappa(B)$, we have
\[
\frac{1}{1- \|B^{-1} \Delta B\|} \le \frac{1}{1-\eps_M \widetilde c_1 \kappa(B)} .
\]
Therefore,
\begin{align}
\|e\| & \leq \eps_M \frac{ \|B^{-1}\|}{1-\eps_M \widetilde c_1 \kappa(B)}  \Big(  (\CS  + \widetilde c_1) \|B\| \|x\|  +
\widetilde c_2  \|B\| \| \widehat y\|  \Big)\nonumber \\
& = \eps_M \frac{ \CS  + \widetilde c_1 }{1-\eps_M \widetilde c_1 \kappa(B)}  \kappa(B) \|x\|  +
 \eps_M \frac{ \widetilde c_2 }{1-\eps_M \widetilde c_1 \kappa(B)}  
 \kappa(B) \| \widehat y\| \label{SM-abs-ferr-exact:eq}
\end{align}
The above bound involves the exact solution $x$ and is therefore not directly computable in general. To make it computable, we can use $x = \widehat x - e$ which gives $\|x\| \leq \|\widehat x\| + \|e\|$, resulting in
\begin{align*}
\big(  \frac{1- \eps_M \kappa(B) (2 \widetilde c_1  + \CS) }{1-\eps_M \widetilde c_1 \kappa(B)} \big) \|e\|& \leq  \eps_M \frac{ \CS  + \widetilde c_1 }{1-\eps_M \widetilde c_1 \kappa(B)}  \kappa(B) \|\widehat x\|  +
 \eps_M \frac{ \widetilde c_2 }{1-\eps_M \widetilde c_1 \kappa(B)}  
 \kappa(B) \| \widehat y\| 
\end{align*}
Thus, by defining $c_1 := \frac{\CS + \widetilde c_1 }{1-\eps_M (2\widetilde c_1 +\CS) \kappa(B)}$ and $c_2$ as in the statement of the theorem we obtain
\[
\| e \|
\le \eps_M c_1 \kappa(B)  \|\widehat x\| + \eps_M c_2 \kappa(B)  \|\widehat y\|
\]
which, with $c_1 < c_2$, gives~\eqref{eq:SM_FerrBnd}.
\end{proof}

In view of~\eqref{SM-abs-ferr-exact:eq}, one way to obtain a relative forward error bound is defining 
$\widecheck c_1 := \frac{ \CS  + \widetilde c_1 }{1-\eps_M \widetilde c_1 \kappa(B)}$ and $\widecheck c_2 := \frac{ \widetilde c_2 }{1-\eps_M \widetilde c_1 \kappa(B)}$
which yields
\[
\frac{\| \widehat x-x \|}{\|x\|}
\le \eps_M \kappa(B) \big( \widecheck c_1 + \widecheck c_2 \frac{\|\widehat y\|}{\|x\|} \big).
\]
From $\widecheck c_1 < \widecheck c_2$, we then obtain
\begin{equation}
\label{rel-ferr-SM:eq}
\frac{\| \widehat x-x \|}{\|x\|}
\le  \eps_M \ \kappa(B) \ \widecheck c_2\ \big(1 +  \frac{\|\widehat y\|}{\|x\|} \big).
\end{equation}
This could be made directly computable in different ways if we have a lower bound for $\|x\|_2$. From $\|b\|=\|Bx\|\leq \|B\| \|x\|$, we obtain our first option 
\begin{equation}
\label{SM_relative_ferr0:eq}
\frac{\| \widehat x-x \|}{\|x\|}
\le \eps_M \kappa(B) \widecheck c_2 \Big( 1 + \|B\| \frac{\|\widehat y\|}{\|b\|} \Big).
\end{equation}
Starting over from~\eqref{rel-ferr-SM:eq}, an a posteriori lower bound for $\|x\|$ which is typically sharper than the previous one can be obtained from the computed approximation \(\widehat x\). The residual $ r=b-B\widehat x$ gives $B(x-\widehat x)=r$. If we know (an upper bound for)
$
\|B^{-1}\| = \frac{1}{\sigma_{\min}(B)} := \sigma_{\min}^{-1}(B)
$
then
\[
    \|x-\widehat x\|_2
    \leq
    \|r\|_2\ \sigma_{\min}^{-1}(B).
\]
By the reverse triangle inequality, 
$  \|x\|
    \geq
    \|\widehat x\|-\|x-\widehat x\|$ which yields
$
    \|x\|
    \geq
    \|\widehat x\| - \|r\| \sigma_{\min}^{-1}(B)$. Assuming the right-hand side is positive, we obtain
\begin{equation}
\label{solution-norm-bound:eq}
       \frac{1}{ \|x\|}
    \leq
    \frac{1}{\|\widehat x\|- \|r\| \sigma_{\min}^{-1}(B)}
\end{equation}
resulting in the following computable forward error bound for the SM solution
\begin{equation}
\label{SM_relative_ferr:eq}
\frac{\|\widehat x-x \|}{\|x\|}
\le \eps_M \kappa(B) \widecheck c_2 \Big( 1+ \frac{\|\widehat y\|}{\|\widehat x\| - \|r\| \sigma_{\min}^{-1}(B)} \Big).
\end{equation}

A third option would be to modify the proof a bit as follows: Going back to~\eqref{ferr_SM_proofStep:eq}, we might keep $b$ (instead of $Bx$ in the first term) to obtain
\begin{align}
\frac{\|\widehat x-x \|}{\|x\|} & \leq
 \frac{\eps_M \|B^{-1}\|}{1-\eps_M \widetilde c_1 \kappa(B)}  \Bigg(  \CS \frac{\|b\|}{{\|x\|}}  + \Big(
    (2\CS+1)\|A\| +   (2n+3)\|u\|\,\|v\|  \Big)\ \frac{\| \widehat y\|}{\|x\|} + \widetilde c_1 \|B\| \Bigg) \nonumber\\
& \leq
 \frac{\eps_M \|B^{-1}\|}{1-\eps_M \widetilde c_1 \kappa(B)}  \Bigg(  \frac{\CS  \|b\| 
+  \Big(
    (2\CS+1)\|A\| +   (2n+3)\|u\|\,\|v\|  \Big) \| \widehat y\|
}{{\|\widehat x\| - \|r\| \sigma_{\min}^{-1}(B)}}  + \widetilde c_1 \|B\| \Bigg)    \nonumber\\
& =
 \frac{\eps_M }{1-\eps_M \widetilde c_1 \kappa(B)}  \Bigg( \frac{\CS  \|b\| 
+  \Big(
    (2\CS+1)\|A\| +   (2n+3)\|u\|\,\|v\|  \Big) \| \widehat y\|
}{\sigma_{\min}(B) (\|\widehat x\| - \|r\| \sigma_{\min}^{-1}(B))}  + \widetilde c_1 \kappa(B) \Bigg) \nonumber\\
& =
 \frac{\eps_M }{1-\eps_M \widetilde c_1 \kappa(B)}  \Bigg( \frac{\CS  \|b\| 
+  \Big(
    (2\CS+1)\|A\| +   (2n+3)\|u\|\,\|v\|  \Big) \| \widehat y\|
}{ \|\widehat x\| \sigma_{\min}(B) - \|r\| }  + \widetilde c_1 \kappa(B) \Bigg) \label{sm-ferr-bnd-option3:eq}
  \end{align}

\section{A modified Sherman-Morrison (MSM) method}
In this section we introduce our new algorithm for solving~\eqref{eq:maingoal}. We first explain how to derive MSM in connection to Govaerts' algorithms for solving bordered linear systems and then an alternative derivation based on splitting the right-hand side $b$ in ~\eqref{eq:maingoal}. 

Govaerts~\cite{Govaerts91} proposed three variants of block Gaussian elimination (both with and without iterative refinement) for solving general $(n+1) \times (n+1)$ bordered linear systems
\begin{equation}
\label{eq:bordered-general}
\begin{bmatrix}
A & b \\
c & d
\end{bmatrix}
\begin{bmatrix}
x \\
y
\end{bmatrix} =
\begin{bmatrix}
f \\
g
\end{bmatrix}
\end{equation}
where $b, c^T$ and $f$ are $n \times 1$ vectors and $d, y$ and $g$ are scalars. He also examined error bounds for the computed solutions. 

Govaerts calls his three variants of block Gaussian elimination~\cite{Govaerts91} BED (block elimination Doolittle), BEC (block elimination Crout) and BEM (block elimination mixed). BED corresponds to an LU-like factorization of the block matrix in~\eqref{eq:bordered-general} where the $L$-factor contains ones on the main diagonal and when ones are on the main diagonal of the $U$-factor the resulting algorithm is called BEC. The algorithms then proceed to solve~\eqref{eq:bordered-general} with substitution. BEM is more involved, but its starting point is to use BED to determine the unknown scalar $y$, which is then substituted into BEC to compute the unknown vector $x$. Rather than outlining BEM for solving~\eqref{eq:bordered-general}, let us first describe the connection between~\eqref{eq:bordered-general} and our perturbed linear system~\eqref{eq:maingoal}, and then explain how BEM adjusts to this specific problem.

Following an observation of Bindel~\cite{Bindel15}, the following special case of the above bordered linear system
\begin{equation}
\label{eq:bordered2}
\begin{bmatrix}
A & u \\
v^T & -1
\end{bmatrix}
\begin{bmatrix}
x \\
\theta
\end{bmatrix} =
\begin{bmatrix}
b \\
0
\end{bmatrix}
\end{equation}
corresponds to~\eqref{eq:maingoal} --- the scalar variable $\theta = v^T x$ serves as a convenient intermediate step in the computation and when substituted in the first equation $Ax+u\theta = b$ gives $(A+uv^T)x = b$. Let us point out right away that, in exact arithmetic, the quantity $\theta$ in~\eqnref{eq:bordered2} is equal to $\alpha/\beta$ as computed in Step 3 of  Alg.~\ref{SM:alg}. To see this, note that $x = y - \theta z$ in the SM formula and so
\[
v^T x = v^T y - \theta (v^T z) = \alpha - \theta (\beta - 1) = \theta.
\]

In Algorithm~\ref{BEM:alg}, we outline how BEM is specialized to the problem of solving~\eqnref{eq:bordered2}.
\begin{algorithm}[!h]
\caption{BEM (Block Elimination Mixed) \cite{Govaerts91} adjusted to solving the special bordered system~\eqnref{eq:bordered2}.}
\label{BEM:alg}
\begin{algorithmic}[1]
\STATE Compute $v^T A^{-1}$ by solving $f^T A = v^T$ for $f$. Equivalently, solve $A^T f = v$.
\STATE Compute $\beta_{\rm{BED}} = 1 + f^T u$.
\STATE Compute $\alpha_{\rm{BED}} = f^T b$, and let $\theta_{\rm{BED}} = \frac{\alpha_{\rm{BED}} }{\beta_{\rm{BED}} }$. 
\STATE Solve $Az= u$ for $z$.  
\STATE Compute $\beta_{\rm{BEC}} = 1 + v^T z$.
\STATE Compute $d := b - u\ \theta_{\rm{BED}}$.
\STATE Solve $A w = d$ for $w$.  
\STATE Compute $\theta_1 :=\frac{v^T w - \theta_{\rm{BED}}}{\beta_{\rm{BEC}}}$.  
\STATE Output $x = w - z \theta_1$.  
\end{algorithmic}
\end{algorithm}
Steps 1-3 compute the scalar component $\theta$ by BED on the basis of Govaerts' argument that initially, only $\theta$ is computed accurately. Note that $A^T$ is the coefficient of the linear system in these BED steps. Then, Steps 4-5 compute $z$ and a new $\beta := 1+v^T z$. The core of the algorithm is its self-correction strategy which begins at Step 6, where it forms the residual of the provisional initial approximation --- namely, scalar component $\theta$ as computed in Steps 1-3 and vector component equal to zero. More precisely, the first 3 steps obtain 
\[
\begin{bmatrix}
x_0\\
\theta_0
\end{bmatrix} :=
\begin{bmatrix}
0_{n \times 1} \\
\theta_{\rm{BED}}
\end{bmatrix} 
\]
as a provisional approximation to the solution of~\eqnref{eq:bordered2}. The corresponding residual is
\begin{equation}
\label{eq:bordered_res_BEM}
r = \begin{bmatrix}
b \\
0
\end{bmatrix} - 
\begin{bmatrix}
A & u \\
v^T & -1
\end{bmatrix}
\begin{bmatrix}
0 \\
\theta_{\rm{BED}}
\end{bmatrix}
=
\begin{bmatrix}
b - u \theta_{\rm{BED}} \\
\theta_{\rm{BED}}
\end{bmatrix}.
\end{equation}
BEM then uses a BEC correction step to solve (in Steps 8-9) the bordered linear system with the above residual as its right-hand side. This means that the correction $(\Delta x,\Delta\theta)$ satisfies
\begin{equation}
\label{eq:bordered_corr_BEM}
\begin{bmatrix}
A & u\\
v^T & -1
\end{bmatrix}
\begin{bmatrix}
\Delta x\\
\Delta\theta
\end{bmatrix}
=
\begin{bmatrix}
b-u\theta_{\rm BED}\\
\theta_{\rm BED}
\end{bmatrix}.
\end{equation}
Once solved, the final solution proposed by the algorithm will be 
\[
\begin{bmatrix}
x\\
\theta
\end{bmatrix} = 
\begin{bmatrix}
x_0\\
\theta_0
\end{bmatrix}
+
\begin{bmatrix}
\Delta x\\
\Delta\theta
\end{bmatrix} = 
\begin{bmatrix}
\Delta x\\
\theta_{\rm BED} + \Delta\theta
\end{bmatrix}
\]

Following our notation, we write $z=A^{-1}u$. Let also $d := b - u\ \theta_{\rm{BED}}$ and $w := A^{-1} d$. The first row of~\eqnref{eq:bordered_corr_BEM} is $A \Delta x + u \Delta\theta = d$ which gives
\[
\Delta x=w-z\Delta\theta.
\]
Substituting this into the second row of~\eqnref{eq:bordered_corr_BEM}, i.e., into $v^T \Delta x - \theta_{\rm{BED}} = \Delta\theta$ yields
\[
v^T w - v^T z \Delta\theta - \theta_{\rm{BED}} = \Delta\theta
\]
and so
\[
\Delta\theta=\frac{v^Tw-\theta_{\rm BED}}{1+v^Tz}.
\]
This is the correction to the scalar variable in the bordered system, i.e., the quantity computed in Step 8. The final solution is
\[
x=\Delta x=w-z\Delta\theta.
\]
as computed in Step 9.

\subsection{Modified Sherman-Morrison (MSM)}
We now slightly modify Alg.~\ref{BEM:alg} to introduce MSM in Alg.~\ref{MSM:alg}. As outlined above, Alg.~\ref{BEM:alg} uses solvers with both $A$ and $A^T$, but in MSM we will solely use solvers with $A$. In particular we compute $\theta$, not like Steps 1-3 of Alg.~\ref{BEM:alg} with a solver for $A^T$, but rather with a solver for $A$. Govaerts' justification of using $A^T$ for computing $\theta$ comes from his observation that when a preconditioned conjugate gradient method is used as a general black-box solver for linear systems, the computed value of $\theta$ was more accurate than the quantity obtained with $A$ (see the first row in \cite[Tab. 1]{Govaerts91}). Govaerts explains \cite[sec. 3]{Govaerts91} that Steps 1-3 could be replaced by any method that produces the scalar unknown $\theta$ accurately. As we use (sparse) Gaussian elimination, instead of computing $\theta = v^T A^{-1} u$ with a solver for $A^T$, we compute it with a solver for $A$ which will be required in other steps of the algorithm as well. 

\begin{algorithm}[!h]
\caption{MSM: Modified Sherman-Morrison formula for solving $(A+uv^T)x=b$.}
\label{MSM:alg}
\begin{algorithmic}[1]
\STATE Solve $Ay = b$ for $y$. 
\STATE Solve $Az= u$ for $z$.  
\STATE Compute $\alpha = v^T y$, the capacitance $\beta = 1+v^T z$, and $\theta = \alpha / \beta$. 
\STATE Compute the new right-hand side $d = b - u\ \theta$.
\STATE Solve $A w = d$ for $w$.  
\STATE Compute $\theta_1 :=(v^T w - \theta) / \beta$.
\STATE Compute the correction vector $p :=z \theta_1$.
\STATE Output $x = w - p$.  
\end{algorithmic}
\end{algorithm}

Starting from $(A+uv^T)x = Ax +u v^Tx$, we treat $\theta = v^T x$ as a second unknown, rewriting the problem as $A x + u \theta = b$. The quantity $\theta$ computed in Step 3 of Alg.~\ref{MSM:alg} is exactly how it was computed in the standard SM; see Step 3 in Alg.~\ref{SM:alg}. Denoting it with $\theta_{\rm{SM}}$, we then take $(x, \theta) \approx (0, \theta_{\rm{SM}})$ as an approximate solution to find the residual
\[
d = b - A 0 - u \theta_{\rm{SM}} = b - u \theta_{\rm{SM}}
\]
which is then used to find the pair $(\delta x, \delta \theta)$ satisfying the pair of correction equations
\begin{equation}
\label{MSM_equations:eq}
A  \delta x + u \delta \theta = d,\qquad v^T \delta x - \delta \theta = \theta_{\rm{SM}}
\end{equation}
such that the final solution pair is obtained as
\[
(x, \theta) = (0, \theta_{\rm{SM}}) + (\delta x, \delta \theta) = (\delta x, \ \theta_{\rm{SM}} + \delta \theta)
\]
The first equation in~\eqnref{MSM_equations:eq} gives $\delta x= A^{-1} d - A^{-1} u \delta\theta$. Denoting $w := A^{-1} d$ and using $z = A^{-1}u$ from the standard SM, this yields
\[
\delta x= w - z \delta\theta.
\]
Substituting this into the second correction equation i.e., into $v^T \delta x - \theta_{\rm{SM}} = \delta\theta$ yields
\[
v^T w - v^T z \delta\theta - \theta_{\rm{SM}} = \delta\theta
\]
and so
\[
\delta\theta=\frac{v^Tw-\theta_{\rm SM}}{1+v^Tz}.
\]
This is the quantity computed in Step 6\footnote{In Alg.~\ref{MSM:alg}, we use $\theta_1$ rather than $\delta \theta$ or $\Delta \theta$, to avoid clashing with the notation of the next section, where we bound the rounding errors in $\theta$ and $\theta_1$ themselves.} of Alg.~\ref{MSM:alg}. The corrected solution $x$ is then 
\[
x = \delta x = w - z \ \delta \theta
\]
which is computed in Steps 7-8.

\subsection{An alternative way to think about MSM}
Solving $(A+uv^T)x=b$ by SM, i.e., by
\[
x = (A+uv^T)^{-1} b = A^{-1}b - A^{-1}u \frac{v^T\! A^{-1}b}{1 + v^T\!A^{-1}u} =: y - z \theta
\]
presents $x$ as a perturbation of $y$ in the direction of $z$, and as $z = A^{-1}u$, we could see the role of $u$ in the formula also as ``directional''. With this in mind, let us split $b$ into two pieces $b = d+ f$ and solve $(A+uv^T)x= d + f$ by linearity 
\[
x = (A+uv^T)^{-1} d + (A+uv^T)^{-1} f
\]

One could then make many choices for $d$ and $f$. Suppose we already have an approximation $\widehat \theta \approx \theta = \frac{v^T y}{1+v^T z}$ and take $d := b - u \widehat \theta$ and $f := u \widehat \theta$. As we see below, this choice of $d$ and $f$ simplifies the linear system solves. 
Applying standard SM to the first system $(A+uv^T) x_d = d$, i.e.,
\[
x_d = A^{-1} d - A^{-1}u(1 + v^T\!A^{-1}u)^{-1}v^T\! A^{-1}d =: w - z \frac{v^T w}{1+v^T z},
\]
we carry out the following steps:
\begin{enumerate}
\item Compute $d := b - u \widehat \theta$.
\item Solve $A w = d$ for $w$.
\item Reuse $z$ if already available, otherwise solve $Az = u$ to find $z$.
\item Compute $\alpha_{d} := v^T w$. 
\item Reuse $\beta$ if already available, otherwise compute $\beta=1+v^T z$. 
\item Compute $\theta_{d} := \alpha_{d} / \beta$.
\item Compute $x_{d} = w - z \theta_d$.
\end{enumerate}

Next, we inspect how SM works for the second problem $x_f = (A+uv^T)^{-1} u \widehat \theta$. As the right-hand side vector $u \widehat \theta$ is essentially already present in the coefficient matrix as well, we could expect some simplifications. Indeed, we have
\begin{align*}
x_f & = \Big(A^{-1} u - A^{-1}u(1 + v^T\!A^{-1}u)^{-1}v^T\! A^{-1}u \Big) \widehat \theta =: \big( z- z \frac{v^T z }{1+v^T z} \big) \widehat \theta \\
& = \big( \frac{z (1+v^T z) -z v^T z}{1+v^T z}\big)\widehat \theta = \frac{z}{1+v^T z} \widehat \theta = z \frac{\widehat \theta}{\beta}
\end{align*}
which shows no extra work is involved in solving the second linear system if $z$ is already available. Putting everything together yields
\[
x = x_d + x_f = (w - z \theta_d) + z \frac{\widehat \theta}{\beta} = w - z \Big( \theta_d - \frac{\widehat \theta}{\beta} \Big) =
w - z \Big( \frac{\alpha_{d}}{\beta} - \frac{\widehat \theta}{\beta} \Big) 
= w - z \Big( \frac{v^T w -\widehat \theta}{\beta} \Big) 
\]
whose form 
\[
x = w - z \theta_{1} \qquad \mbox{ where }\qquad \theta_{1} := \frac{v^T w -\widehat \theta}{\beta}
\]
resembles that of the standard SM applied to the original problem. From among the seven steps listed above, we can avoid the last two and instead perform the following:
 \begin{enumerate}[start=6]
\item Compute $\theta_{1} := \frac{\alpha_{d} - \widehat \theta}{\beta} = \frac{v^T w - \widehat \theta}{\beta} $.
\item Compute the correction vector $p := z \theta_1$.
\item Compute $x = w - p$
\end{enumerate}
which is another way to derive Alg.~\ref{MSM:alg}.

\section{Error bounds for MSM}
In this section we establish backward and forward error bounds for Alg.~\ref{MSM:alg}. We will discuss and illustrate these bounds in the sequel.

\subsection{Backward error of MSM}
We begin with a backward error representation for MSM which is the foundation of our subsequent backward and forward error bounds.
\begin{theorem}
\label{MSM-back-err:thm}
Assuming $\widehat\beta\neq 0$, the computed solution $\widehat x$ of MSM satisfies 
\begin{equation}
    \left(A+\Delta A+(u+\Delta u)(v+\Delta v)^T\right)\widehat x 
    =
    b+\Delta b + \Delta G \ \widehat p + \Delta g
\label{eq:BE}
\end{equation}
where 
\begin{align}
 &   \|\Delta A\|
    \leq
    \big(        (1+\CS)\eps_M + \mathcal O (\eps_M^2)    \big) \|A\|
\label{eq:DeltaA-bound}\\
&    \|\Delta u\|
    \leq
   \big(2 \eps_M + \mathcal O(\eps_M^2) \big) \|u\|
\label{eq:Deltau-bound}\\
&  \|\Delta v\|
    \leq
       \big((n+1) \eps_M + \mathcal O(\eps_M^2) \big)\|v\|
\label{eq:Deltav-bound}\\
&    \|\Delta b\|    \leq  \eps_M\|b\|
\label{eq:Deltab-bound-raw}\\
&  \|\Delta G\| \leq \big(
    (6+2\CS) \eps_M + \mathcal O(\eps_M^2)
\big)     \|A\|                                           
+
\big(
(2n+4) \eps_M + \mathcal O(\eps_M^2) \big) \|u\|\,\|v\|
\label{eq:DeltaG-bound2}    \\
&    \|\Delta g\| \leq 
\CS \eps_M \Big( \|b\| + \|u\| (|\widehat \theta| + |\widehat \theta_1|) \Big) + \mathcal O (\eps_M^2)
\label{eq:Delta-g-bound}
\end{align}
and $\widehat p$ is the correction vector computed in Step 7, and $\widehat \theta$ and $\widehat \theta_1$ are the scalars computed in Steps 3 and 6 of the algorithm, all in floating point arithmetic.
\end{theorem}

\begin{proof}
The proof is structured as follows. First, we derive a backward error representation for each of the eight steps of Alg.~\ref{MSM:alg}. We then combine these individual representations to obtain the overall representation~\eqref{eq:BE}, which constitutes the longest part of the proof. Finally, we bound each perturbation to the input data, thereby establishing~\eqref{eq:DeltaA-bound}--\eqref{eq:Delta-g-bound} in turn.\ The backward stability of the first two solves in Steps~1 and~2 gives
\begin{equation}
    (A+\Delta_1A)\widehat y=b+\Delta_1b,
\label{eq:y-solve}
\end{equation}
and
\begin{equation}
    (A+\Delta_2A)\widehat z=u+\Delta_1u,
\label{eq:z-solve}
\end{equation}
where
\[
    \|\Delta_1A\|\leq \CS\eps_M\|A\|,
    \qquad
    \|\Delta_1b\|\leq \CS\eps_M\|b\|,
\]
and
\[
    \|\Delta_2A\|\leq \CS\eps_M\|A\|,
    \qquad
    \|\Delta_1u\|\leq \CS\eps_M\|u\|.
\]

The two inner products in Step 3 give
\begin{equation}
    \widehat\alpha=v^T\Theta_1\widehat y,
    \qquad
    \Theta_1=I+D_y,
    \qquad
    |D_y|\leq \gamma_n I,
\label{eq:alpha}
\end{equation}
and
\begin{equation}
    \widetilde\beta=v^T\Theta_2\widehat z,
    \qquad
    \Theta_2=I+D_z,
    \qquad
    |D_z|\leq \gamma_n I.
\label{eq:beta}
\end{equation}
The scalar $\widehat\theta$ satisfies
\begin{equation}
    \widehat\theta
    = 
    fl(\frac{\widehat \alpha}{\widehat\beta}) = \frac{\widehat \alpha}{\widehat\beta} \langle 1\rangle 
    = \frac{\widehat\alpha}{1\oplus\widetilde\beta} \langle 1\rangle =
    \frac{\widehat\alpha}{1+\widetilde\beta} \langle 2\rangle.
\label{eq:theta}
\end{equation}

\noindent Step 4 gives
\begin{equation}
    \widehat u_2 := fl(u \widehat \theta) = \Theta_3u\,\widehat\theta,
    \qquad
    \Theta_3=I+D_3,
    \qquad
    |D_3|\leq \eps_M I
\label{eq:u2}
\end{equation}
and 
\begin{equation}
    \widehat d=\Theta_4(b-\widehat u_2),
    \qquad
    \Theta_4=I+D_4,
    \qquad
    |D_4|\leq \eps_M I.
\label{eq:b1-step}
\end{equation}
Substituting \eqref{eq:u2} into \eqref{eq:b1-step}, we obtain
\begin{equation}
    \widehat d
    =
    \Theta_4b-\Theta_4\Theta_3u\,\widehat\theta .
\label{eq:b1-expanded}
\end{equation}
Step 5 solves a third linear system whose backward stability gives
\begin{equation}
    (A+\Delta_3A)\widehat w
    =
    \widehat d+\Delta d,
\label{eq:w-solve}
\end{equation}
where
\begin{equation}
    \|\Delta_3A\|\leq \CS\eps_M\|A\|,
    \qquad
    \|\Delta d\|\leq \CS\eps_M\|\widehat d\|.
\label{eq:Deltab1-bound}
\end{equation}
Using \eqref{eq:b1-expanded}, equation \eqref{eq:w-solve} becomes
\begin{equation}
    (A+\Delta_3A)\widehat w
    =
    \Theta_4b-\Theta_4\Theta_3u\,\widehat\theta+\Delta d.
\label{eq:w-solve-expanded}
\end{equation}

\noindent Considering the operations in Step 6 one at a time gives
\begin{equation}
    \widehat\lambda := fl(v^T \widehat w) = v^T\Theta_5\widehat w,
    \qquad
    \Theta_5=I+D_w,
    \qquad
    |D_w|\leq \gamma_n I
\label{eq:lambda}
\end{equation}
and 
\begin{equation}
    \widehat\theta_1
    =
    \frac{\widehat\lambda-\widehat\theta}{1+ \widetilde\beta} \langle 3\rangle.
\label{eq:theta1}
\end{equation}
Therefore
\begin{equation}
    \widehat\lambda-\widehat\theta
    =
    \langle 3\rangle(1+\widetilde\beta)\widehat\theta_1
\label{eq:lambda-theta}
\end{equation}
which has used $\frac{1}{\langle 3\rangle} = \langle 3\rangle$; see Section 1. Using \eqref{eq:beta} and \eqref{eq:lambda}, this becomes
\[
    v^T\Theta_5\widehat w-\widehat\theta
    =
    \langle 3\rangle\left(1+v^T\Theta_2\widehat z\right)\widehat\theta_1.
\]
Equivalently,
\begin{equation}
    \widehat\theta+\langle 3\rangle\widehat\theta_1
    =
    v^T\Theta_5\widehat w
    -
    \langle 3\rangle v^T\Theta_2\widehat z\,\widehat\theta_1 .
\label{eq:theta-identity-pre}
\end{equation}

\noindent Step 7 gives
\begin{equation}
    \widehat p=\Theta_6\widehat z\,\widehat\theta_1,
    \qquad
    \Theta_6=I+D_6,
    \qquad
    |D_6|\leq \eps_M I
\label{eq:z3}
\end{equation}
and finally Step 8 gives
\begin{equation}
    \widehat x=\Theta_7(\widehat w-\widehat p),
    \qquad
    \Theta_7=I+D_7,
    \qquad
    |D_7|\leq \eps_M I.
\label{eq:x-step}
\end{equation}
Using \eqref{eq:z3}, we have
\[
    \widehat x
    =
    \Theta_7(\widehat w-\Theta_6\widehat z\,\widehat\theta_1).
\]
Since $\Theta_7$ is nonsingular,
\begin{equation}
    \widehat w
    =
    \Theta_7^{-1}\widehat x
    +
    \Theta_6\widehat z\,\widehat\theta_1 .
\label{eq:w-in-terms-of-x}
\end{equation}
Substituting \eqref{eq:w-in-terms-of-x} into
\eqref{eq:theta-identity-pre} gives
\begin{align}
    \widehat\theta+\langle 3\rangle\widehat\theta_1
    &=
    v^T\Theta_5
    \left(
        \Theta_7^{-1}\widehat x
        +
        \Theta_6\widehat z\,\widehat\theta_1
    \right)
    -
    \langle 3\rangle v^T\Theta_2\widehat z\,\widehat\theta_1      \nonumber \\
    &=
    v^T\Theta_5\Theta_7^{-1}\widehat x
    +
    v^T(\Theta_5\Theta_6-\langle 3\rangle\Theta_2)\widehat z\,
    \widehat\theta_1 .\label{eq:theta-identity}
\end{align}
Defining $s:=    v^T(\Theta_5\Theta_6-\langle 3\rangle\Theta_2)\widehat z$ gives
\begin{equation}
    \widehat\theta+\langle 3\rangle\widehat\theta_1
    =
    v^T\Theta_5\Theta_7^{-1}\widehat x+s\widehat\theta_1 .
\label{eq:theta-identity-s}
\end{equation}

We now substitute \eqref{eq:w-in-terms-of-x} into
\eqref{eq:w-solve-expanded}. This gives
\begin{equation}
    (A+\Delta_3A)\Theta_7^{-1}\widehat x
    +
    (A+\Delta_3A)\Theta_6\widehat z\,\widehat\theta_1
    +
    \Theta_4\Theta_3u\,\widehat\theta                  
    =
    \Theta_4b+\Delta d .
\label{eq:x-key}
\end{equation}
We rewrite the two terms involving $\widehat \theta$ and $\widehat \theta_1$ by adding and subtracting
$\langle 3\rangle\Theta_4\Theta_3u\,\widehat\theta_1$:
\[
\begin{aligned}
    \Theta_4\Theta_3u\,\widehat\theta
    +
    (A+\Delta_3A)\Theta_6\widehat z \widehat\theta_1
    &=
    \Theta_4\Theta_3u(\widehat\theta+\langle 3\rangle\widehat\theta_1)
    +
    \big((A+\Delta_3A)\Theta_6\widehat z - \langle 3\rangle\Theta_4\Theta_3u \big)\widehat\theta_1 .
\end{aligned}
\]
Using \eqref{eq:theta-identity-s},
\begin{align}
    \Theta_4\Theta_3u\,\widehat\theta
    +
    (A+\Delta_3A)\Theta_6\widehat z \widehat\theta_1
    &=
    \Theta_4\Theta_3u\,v^T\Theta_5\Theta_7^{-1}\widehat x
    +
    \Theta_4\Theta_3u\,s\widehat\theta_1      \nonumber           \\
    &\quad
    +
    \Big( (A+\Delta_3A)\Theta_6\widehat z -\langle 3\rangle\Theta_4\Theta_3u \Big)\widehat\theta_1       \nonumber      \\
    &=
    \Theta_4\Theta_3u\,v^T\Theta_5\Theta_7^{-1}\widehat x
    +
    h\widehat\theta_1, \label{eq:h-short-derivation}
\end{align}
where
\begin{equation}
    h:=(A+\Delta_3A)\Theta_6\widehat z -\langle 3\rangle\Theta_4\Theta_3u+\Theta_4\Theta_3u\,s .
\label{eq:h-short}
\end{equation}
Substituting the definition of $s$ gives
\begin{equation}
    h
    =
    (A+\Delta_3A)\Theta_6\widehat z
    -
    \langle 3\rangle\Theta_4\Theta_3u
    +
    \Theta_4\Theta_3u\,
    v^T(\Theta_5\Theta_6-\langle 3\rangle\Theta_2)\widehat z .
\label{eq:h-before-zsolve}
\end{equation}

From the solve for $\widehat z$, we know that $(A+\Delta_2A)\widehat z=u+\Delta_1u$. Hence
\[
    u=(A+\Delta_2A)\widehat z-\Delta_1u.
\]
Substituting this into \eqref{eq:h-before-zsolve}
yields
\[
\begin{aligned}
h
&=
(A+\Delta_3A)\Theta_6\widehat z
-
\langle 3\rangle\Theta_4\Theta_3
\left((A+\Delta_2A)\widehat z-\Delta_1u\right)       +
\Theta_4\Theta_3u\,
v^T(\Theta_5\Theta_6-\langle 3\rangle\Theta_2)\widehat z          \\
&=
\left[
(A+\Delta_3A)\Theta_6
-
\langle 3\rangle\Theta_4\Theta_3(A+\Delta_2A)
\right]\widehat z      +
\langle 3\rangle\Theta_4\Theta_3\Delta_1u                                 \\
&\quad
+
\Theta_4\Theta_3u\,
v^T(\Theta_5\Theta_6-\langle 3\rangle\Theta_2)\widehat z .
\end{aligned}
\]
With the definitions
\[
\Delta G_1 :=
- (A+\Delta_3A)\Theta_6 + \langle 3\rangle\Theta_4\Theta_3(A+\Delta_2A) - \Theta_4\Theta_3u\, v^T(\Theta_5\Theta_6-\langle 3\rangle\Theta_2),
\]
and $\widehat p = \Theta_6 \widehat z \widehat \theta_1$, which is the quantity computed in Step 7 of the algorithm, we see that 
\[
h \widehat \theta_1 = 
- \Delta G_1 \widehat z \widehat\theta_1 + \langle 3\rangle\Theta_4\Theta_3\Delta_1u \widehat \theta_1 
= - \Delta G_1 \Theta_6^{-1} \Theta_6 \widehat z \widehat\theta_1 + \langle 3\rangle\Theta_4\Theta_3\Delta_1u \widehat \theta_1 
=: 
- \Delta G \widehat p + \langle 3\rangle\Theta_4\Theta_3\Delta_1u \widehat \theta_1
\]
where
\begin{equation}
\Delta G := \Delta G_1 \Theta_6^{-1}.
\label{eq:Hmat-def}
\end{equation}

Substituting \eqref{eq:h-short-derivation} into \eqref{eq:x-key} gives
\[
    (A+\Delta_3A)\Theta_7^{-1}\widehat x
    +
    \Theta_4\Theta_3u\,v^T\Theta_5\Theta_7^{-1}\widehat x
    - \Delta G \widehat p + \langle 3\rangle\Theta_4\Theta_3\Delta_1u \widehat \theta_1
    =
    \Theta_4b+\Delta d.
\]
Therefore
\begin{equation}
    \left[
        (A+\Delta_3A)\Theta_7^{-1}
        +
        \Theta_4\Theta_3u\,v^T\Theta_5\Theta_7^{-1}
    \right]\widehat x
    =
    \Theta_4b+\Delta d + \Delta G \widehat p  - \langle 3\rangle\Theta_4\Theta_3\Delta_1u \widehat \theta_1.
\label{eq:almost-final}
\end{equation}

We now identify the perturbations $\Delta A, \Delta u, \Delta v, \Delta b$ and $\Delta g$. For the first term in the bracket above, we add and subtract $A + \Delta_3A$ to get
\[
    (A+\Delta_3A)\Theta_7^{-1}
    =
    A+\Delta_3A+(A+\Delta_3A)(\Theta_7^{-1}-I) 
    =:    A+\Delta A,
\]
which gives
\begin{equation}
\label{eq:DeltaA-def}
    \Delta A =
    \Delta_3 A+(A+\Delta_3A)(\Theta_7^{-1}-I).
\end{equation}
For the second term in~\eqref{eq:almost-final}, let
\[
    \Theta_4\Theta_3u\,v^T\Theta_5\Theta_7^{-1}
    =:    (u+\Delta u)(v+\Delta v)^T,
\]
where
\begin{equation}
    \Delta u
    :=
    (\Theta_4\Theta_3-I)u,
\label{eq:Deltau-def}
\end{equation}
and $(v+\Delta v)^T :=v^T\Theta_5\Theta_7^{-1}$ which gives
\begin{equation}
    \Delta v
    =
    \left(\Theta_7^{-T}\Theta_5^T-I\right)v,
\label{eq:Deltav-def}
\end{equation}

Defining
\begin{equation}    
\label{eq:Deltab-new-def}
\Delta b :=(\Theta_4-I)b
\end{equation}
and
\[
  \Delta g :=   \Delta d  - \langle 3\rangle\Theta_4\Theta_3\Delta_1u \widehat \theta_1
  \]
we see that~\eqref{eq:almost-final} finally gives the desired error representation~\eqref{eq:BE}.

It remains to bound each of the perturbations $\Delta A, \Delta u, \Delta v, \Delta b, \Delta G$ and $\Delta g$. 

\paragraph{Bounding \(\|\Delta A\|\)}
Since
$  \Theta_7=I+D_7$ with $ |D_7|\leq \eps_M I$, in view of~\eqref{eq:Theta-I-bnd} we have $\|\Theta_7^{-1}-I\|\leq \gamma_1$. Therefore~\eqref{eq:DeltaA-def} yields
\[
\begin{aligned}
    \|\Delta A\|
    &\leq
    \|\Delta_3A\|
    +
    \|A+\Delta_3A\|\,\|\Theta_7^{-1}-I\|      
    \leq
    \CS\eps_M\|A\|
    +
    (1+\CS\eps_M)\|A\|\gamma_1                  \\
    &=
    \left[
        \CS\eps_M+(1+\CS\eps_M)\gamma_1
    \right]\|A\|
\end{aligned}
\]
which proves \eqref{eq:DeltaA-bound}.

\paragraph{Bounding \(\|\Delta u\|\)}
Next, $    \Theta_4\Theta_3$ has diagonal entries of the form $\langle 2\rangle$. Hence
$\|\Theta_4\Theta_3-I\|\leq \gamma_2$. Therefore, from~\eqref{eq:Deltau-def} we see that
\[
    \|\Delta u\|
    =
    \|(\Theta_4\Theta_3-I)u\|
    \leq
    \gamma_2\|u\|
\]
which yields~\eqref{eq:Deltau-bound}.

\paragraph{Bounding \(\|\Delta v\|\)}
Also, $  \Theta_5\Theta_7^{-1}$ has diagonal entries of the form $\langle n+1\rangle$. Hence
\[
    \|\Theta_5\Theta_7^{-1}-I\|
    \leq
    \gamma_{n+1}.
\]
Taking norms from both sides of~\eqref{eq:Deltav-def} yields
$
    \|\Delta v\|
    \leq
    \gamma_{n+1}\|v\|
$
which proves \eqref{eq:Deltav-bound}.

\paragraph{Bounding \(\|\Delta b\|\)}
For $\Delta b$, taking norms from both sides of~\eqref{eq:Deltab-new-def} readily gives
$\|\Delta b\| \leq \|\Theta_4-I\|\|b\| \leq    \eps_M\|b\|$, which is~\eqref{eq:Deltab-bound-raw}.

\paragraph{Bounding \(\|\Delta G\|\)}
Recall from~\eqref{eq:Hmat-def} that $\Delta G = \Delta G_1 \Theta_6^{-1}$ where
\[
\Delta G_1 := - (A+\Delta_3A)\Theta_6 + \langle 3\rangle\Theta_4\Theta_3(A+\Delta_2A) - \Theta_4\Theta_3u\, v^T(\Theta_5\Theta_6-\langle 3\rangle\Theta_2).
\]
For the first two terms in $\Delta G_1$, we have
\begin{align}
&- (A+\Delta_3A)\Theta_6
+
\langle 3\rangle\Theta_4\Theta_3(A+\Delta_2A)                       \nonumber  \\
&\quad
=
-A\Theta_6 + \langle 3\rangle\Theta_4\Theta_3 A
-
\Delta_3A\,\Theta_6
+
\langle 3\rangle\Theta_4\Theta_3\Delta_2A.
\label{new-deltaG-1:eq}
\end{align}
Since $\|I - \Theta_6\| = \| D_6\| \leq \eps_M$ and since $\langle 3\rangle\Theta_4\Theta_3$ has diagonal entries of the form \(\langle 5\rangle\), we have
\[
    \|\langle 3\rangle\Theta_4\Theta_3-I\|\leq \gamma_5.
\]
Therefore
\[
    \|-\Theta_6 + \langle 3\rangle\Theta_4\Theta_3\|
    \leq
    \|I - \Theta_6\|+\|\langle 3\rangle\Theta_4\Theta_3-I\| 
    \leq
    \eps_M+\gamma_5 .
\]
Moreover,
\[
    \|\Theta_6\|\leq 1+\eps_M,
    \qquad
    \|\langle 3\rangle\Theta_4\Theta_3\|\leq 1+\gamma_5.
\]
Adding and subtracting $A$ in the first two terms of~\eqref{new-deltaG-1:eq}, taking norms and using
\[
    \|\Delta_3A\|\leq \CS\eps_M\|A\|,
    \qquad
    \|\Delta_2A\|\leq \CS\eps_M\|A\|,
\]
we obtain
\begin{align}
&\left\|
- (A+\Delta_3A)\Theta_6
+
\langle 3\rangle\Theta_4\Theta_3(A+\Delta_2A)
\right\|                                                    \nonumber    \\
&\quad
\leq
(\eps_M+\gamma_5)\|A\|
+
\CS\eps_M(1+\eps_M)\|A\|
+
\CS\eps_M(1+\gamma_5)\|A\|. \label{eq:first-h-bound}
\end{align}

For the last term in $\Delta G_1$, observe that $\Theta_5\Theta_6$ has diagonal entries of the form \(\langle n+1\rangle\), whereas $ \langle 3\rangle\Theta_2$ has diagonal entries of the form \(\langle n+3\rangle\). Hence
\[
    \|\Theta_5\Theta_6-I\|\leq \gamma_{n+1},
    \qquad
    \|\langle 3\rangle\Theta_2-I\|\leq \gamma_{n+3}.
\]
Therefore
\[
    \|\Theta_5\Theta_6-\langle 3\rangle\Theta_2\|
    \leq
    \|\Theta_5\Theta_6-I\|
    +
    \|\langle 3\rangle\Theta_2-I\|                        
    \leq
    \gamma_{n+1}+\gamma_{n+3}.
\]
Since
\[
    \|\Theta_4\Theta_3u\|\leq (1+\gamma_2)\|u\|,
\]
we get
\begin{align}
\left\|
\Theta_4\Theta_3u\,
v^T(\Theta_5\Theta_6-\langle 3\rangle\Theta_2)\right\|  \leq  
 (1+\gamma_2)\|u\|\, \|v\|\,
\|\Theta_5\Theta_6-\langle 3\rangle\Theta_2\|\,                                            \nonumber \\
\leq
(1+\gamma_2)(\gamma_{n+1}+\gamma_{n+3})
\|u\|\,\|v\|\,. 
\label{eq:third-h-bound}
\end{align}
Combining~\eqref{eq:first-h-bound} and~\eqref{eq:third-h-bound}, we obtain
\[
\| \Delta G_1\| \leq 
\Big(
    (\eps_M+\gamma_5)
    +
    \CS\eps_M(1+\eps_M)
    +
    \CS\eps_M(1+\gamma_5)
\Big)     \|A\|                                           
+
(1+\gamma_2)(\gamma_{n+1}+\gamma_{n+3})
\|u\|\,\|v\|.
\label{eq:DeltaG-bound}
\]
which gives~\eqref{eq:DeltaG-bound2}. Note that by~\eqref{eq:Theta-inv-bnd}, we have $\|\Theta_6^{-1} \| \le 1 + \eps_M + \mathcal{O}(\eps_M^2)$ implying that the above bound for $\| \Delta G_1\|$ carries over to $\| \Delta G\|$ up to $\eps_M^2$.

\paragraph{Bounding \(\|\Delta g\|\)}
To bound $\|\Delta g\| = \|\Delta d - \langle 3\rangle\Theta_4\Theta_3\Delta_1u \widehat \theta_1 \|$, first recall from~\eqref{eq:Deltab1-bound} that 
\[
    \|\Delta d\|\leq \CS\eps_M\|\widehat d\|
\]
where, by~\eqref{eq:b1-expanded}, we have $\widehat d = \Theta_4(b-\Theta_3u\,\widehat\theta)$. Therefore
\[
\begin{aligned}
    \|\widehat d\|
    &\leq
    \|\Theta_4\|
    \left(
        \|b\|+\|\Theta_3\|\|u\|\,|\widehat\theta|
    \right)
    \leq
    (1+\eps_M)
    \left(
        \|b\|+(1+\eps_M)\|u\|\,|\widehat\theta|
    \right)
    \end{aligned}
\]
resulting in
\begin{equation}
    \|\Delta d\|\leq \CS \eps_M \    (1+\eps_M)
    \left(
        \|b\|+(1+\eps_M)\|u\|\,|\widehat\theta|
    \right).
\label{eq:Delta-b1-bound-new}
\end{equation}

To bound the second term of $\Delta g$, recall that $\langle 3\rangle\Theta_4\Theta_3$ has diagonal entries of the form \(\langle 5\rangle\) implying that $\|\langle 3\rangle\Theta_4\Theta_3\|\leq 1+\gamma_5$. 
Therefore
\begin{equation}
 \| - \langle 3\rangle\Theta_4\Theta_3\Delta_1u \ \widehat \theta_1\|
    \leq
    (1+\gamma_5)\|\Delta_1u\|  \     |\widehat \theta_1|
    \leq
    \CS\eps_M(1+\gamma_5)\|u\|\ |\widehat \theta_1|.
\label{eq:second-h-bound}
\end{equation}
Combining~\eqref{eq:Delta-b1-bound-new} and~\eqref{eq:second-h-bound} gives
\[
    \|\Delta g\|
    \leq
\CS \eps_M \    (1+\eps_M)
    \left(
        \|b\|+(1+\eps_M)\|u\|\,|\widehat\theta|
    \right)
    +
    	\CS\eps_M (1+\gamma_5)\|u\|\ |\widehat \theta_1|
\]
which gives~\eqref{eq:Delta-g-bound}. 
\end{proof}

We are ready to use the previous theorem to find a bound for the MSM backward error.
\begin{theorem} Assume $\widehat x$ denotes the solution of MSM in floating point arithmetic. Then its backward error satisfies
\begin{align}
\omega(\hat{x}) & \leq  \ 
 \eps_M (2n+4 + 2\CS) \frac{\big( \| A\| + \|u\| \|v\| \big)\,(\| \widehat x\| + \| \widehat p\|)}{\|A+u v^{T}\|\, \|\widehat x\| + \|b\|} \nonumber \\ 
 & \ + \eps_M\, (1+\CS)  \frac{\|u\| \big( |\widehat \theta| + |\widehat \theta_1| \big) + \|b\|}{\|A+u v^{T}\|\, \|\widehat x\| + \|b\|} + \mathcal{O}(\eps_M^2).
  \label{MSM_BackErrBnd:eq}
\end{align}
\end{theorem}
\begin{proof}
The representation~\eqref{eq:BE} yields
\begin{equation}
    r
    =
    \Big(\Delta A +  u \Delta v^T +  \Delta u\ v^T + \Delta u\ \Delta v^T\Big)\widehat x -\Delta b - \Delta G \ \widehat p - \Delta g
\end{equation}
Now proceeding similar to Theorem~\ref{thm:new-SM-BackErrBnd}, we take the norm to obtain
\begin{align*}
\| r \| & \leq \Big( \| \Delta A \| + \|u\|\ \|\Delta v\|+ \|\Delta u\| \|v\| + \|\Delta u\| \ \|\Delta v\| \Big) \|\widehat x\| + \|\Delta b\| 
+ \| \Delta G\|\ \|\widehat p\| 
+ \|\Delta g\| \nonumber \\
& \leq 
\eps_M \Big( (1+\CS) \|A\| +  (n+3) \|u\|\ \|v\| \Big) \|\widehat x\| + \eps_M \|b\| + \mathcal O(\eps_M^2) \nonumber \\
& + 
\eps_M \Big(
    (6+2\CS)\|A\|
    +
    (2n+4)\|u\|\,\|v\|  \Big) \ \|\widehat p\|  + \CS \eps_M \Big( \|b\| + \|u\| (|\widehat \theta| + |\widehat \theta_1|) \Big) \nonumber \\
&\leq \eps_M (n+3+\CS) \Big(  \| A\| + \|u\| \|v\| \Big)\,\| \widehat x\| \, +\, \eps_M\, (1+\CS) \|b\| \, \nonumber \\[4pt]
& \ \ \ +\, \eps_M  (2n+4+2\CS) \Big( \|A\| + \|u\| \|v\| \Big)\,\|\widehat p\| + 
\eps_M \CS \|u\| \big( |\widehat \theta| + |\widehat \theta_1| \big) +
\mathcal{O}(\eps_M^2) 
\end{align*}
Hence,
\[
\| r \| \lesssim
\eps_M (2n+4+2\CS) \big(  \| A\| + \|u\| \|v\| \big)\ \big( \| \widehat x\| + \|\widehat p\| \big) 
 + \eps_M\, (1+\CS) \Big( \|b\| + \|u\| \big( |\widehat \theta| + |\widehat \theta_1| \big) \Big)
 \]
where $\lesssim$ means second-order terms in $\eps_M$ are dropped. In view of~\eqref{eq:backErrB_def} we obtain~\eqref{MSM_BackErrBnd:eq}.
\end{proof}

\subsection{Forward error of MSM}
We can follow the same idea as in Theorem~\eqref{thm:SM_FerrBnd} to establish a forward error bound using Theorem~\ref{MSM-back-err:thm}. 
\begin{theorem}
\label{thm:MSM_FerrBnd}
Assume linear systems $Ay=b$, $Az=u$ and $Aw=\widehat d$ are solved normwise backward stably. Define
\begin{align*}
& \widetilde c_3 := (n+\CS+3) \frac{\|A\| + \|u\| \|v\|}{\|B\|}, \ \widetilde c_4 := (2n+2\CS+4) \frac{\|A\| + \|u\| \|v\|}{\|B\|},\\ & \widetilde c_5 := \frac{\CS \|u\|}{\|B\|}, \ c_3 := \frac{1+\CS + \widetilde c_3 }{1-\eps_M \widetilde c_3 \kappa(B)}, \ c_4 := \frac{\widetilde c_4 }{1-\eps_M \widetilde c_3 \kappa(B)}, \ c_5 := \frac{\widetilde c_5 }{1-\eps_M \widetilde c_3 \kappa(B)}
\end{align*}
and assume that $\eps_M \widetilde c_3 \kappa(B)<1$. If $\widehat x$ is computed with MSM in floating point arithmetic, then
\begin{equation}\label{eq:MSM_FerrBnd}
\| \widehat x-x \|
\le \frac{\eps_M \kappa(B) }{1-c_3 \eps_M \kappa(B)} \Big( c_3 \|\widehat x\| + c_4  \|\widehat p\| + c_5 (|\widehat \theta| + |\widehat \theta_1|)  \Big).
\end{equation}
\end{theorem}

\begin{proof}
As the assumptions of Theorem~\ref{MSM-back-err:thm} are satisfied, we substitute $b = (A+uv^T) x$ into~\eqref{eq:BE} to get
\[
    \left(A+\Delta A+(u+\Delta u)(v+\Delta v)^T\right)\widehat x 
    =
    ( A + uv^T) x +\Delta b + \Delta G \ \widehat p + \Delta g
\]
which then by defining $e:= \widehat x - x$, $B:= A+uv^T$ and $\Delta B:= \Delta A + u\ \Delta v^T+\Delta u\ v^T +\Delta u\ \Delta v^T$ and subtracting $\Delta B\ x$ from both sides yield
\[
(B + \Delta B) e =  \Delta b + \Delta G \ \widehat p + \Delta g - \Delta B\ x 
\]
Multiplying by $B^{-1}$ from the left and defining $E:=B^{-1}\Delta B$ gives
\begin{equation}
\label{eq:IplusEe-MSM}
(I+E)e = B^{-1}\bigl(\Delta b +  \Delta G \ \widehat p + \Delta g - \Delta B\ x \bigr).
\end{equation}
The definition of $\widetilde c_3$, together with the bounds~\eqref{eq:DeltaA-bound},~\eqref{eq:Deltau-bound} and~\eqref{eq:Deltav-bound} ---which already hold--- guarantee that $\|\Delta B\| \le \eps_M \widetilde c_3 \|B\|$. Hence, $\|E\|< 1$ and $\|(I+E)^{-1}\| \le \big(1-\|B^{-1}\Delta B\|\big)^{-1}$. 
Hence, multiplying both sides by $(I+E)^{-1}$, taking norms, using the bounds on $\|\Delta b\|$ and $\|\Delta F\|$ from Theorem~\ref{MSM-back-err:thm}, and $b = Bx$ give
\begin{align}
\|e\| 
&\leq \|(I+E)^{-1}\| \|B^{-1}\|
 \Big(
 \|\Delta b\| + \|\Delta G\|\,\|\widehat p\|
 + \|\Delta g\| + \|\Delta B\|\,\|x\|
 \Big) \nonumber \\
&\leq
 \frac{\|B^{-1}\|}{1-\|B^{-1}\Delta B\|}
 \Big(
 \|\Delta b\| + \|\Delta G\|\,\|\widehat p\|
 + \|\Delta g\| + \|\Delta B\|\,\|x\|
 \Big) \nonumber \\
&\leq
 \frac{\eps_M\|B^{-1}\|}
      {1-\eps_M \widetilde c_3 \kappa(B)}
 \Bigg[
 \|b\|
 + \Big( (2\CS+6)\|A\| + (2n+4)\|u\|\,\|v\| \Big)
   \|\widehat p\| \nonumber \\
&\qquad\qquad\qquad
 + \CS\Big( \|b\| + \|u\|(|\widehat\theta|+|\widehat\theta_1|) \Big)
 + \widetilde c_3 \|B\|\,\|x\|
 \Bigg]
 \label{abs-for-err-msm-temp:eq} \\
&\leq
 \frac{\eps_M\|B^{-1}\|}
      {1-\eps_M \widetilde c_3 \kappa(B)}
 \Bigg[
 \|Bx\|
 + (2\CS+2n+4)(\|A\|+\|u\|\,\|v\|)
   \|\widehat p\| \nonumber \\
&\qquad\qquad\qquad
 + \CS\Big( \|Bx\| + \|u\|(|\widehat\theta|+|\widehat\theta_1|) \Big)
 + \widetilde c_3 \|B\|\,\|x\|
 \Bigg]. \nonumber \\
  & =  \frac{\eps_M \|B^{-1}\|}{1-\eps_M \widetilde c_3 \kappa(B)}  \Bigg(  
  (1+\CS) \|B\| \|x\|  + \widetilde c_4  \|B\| \| \widehat p\|  
  + \CS  \|u\| (|\widehat \theta| + |\widehat \theta_1|)
  + \widetilde c_3 \|B\| \|x \| 
\Bigg)\nonumber \\
  & =  \frac{\eps_M \|B^{-1}\|}{1-\eps_M \widetilde c_3 \kappa(B)}  \Bigg(  
  (1+\CS +\widetilde c_3) \|B\| \|x\|  + \widetilde c_4  \|B\| \| \widehat p\|  
  + \CS  \|u\| (|\widehat \theta| + |\widehat \theta_1|) 
\Bigg).\nonumber 
\end{align}
Therefore,
\begin{align*}
\|e\| & \leq \eps_M \frac{ \|B^{-1}\|}{1-\eps_M \widetilde c_3 \kappa(B)}  \Bigg(  (1+\CS  + \widetilde c_3) \|B\| \|x\|  +
\widetilde c_4  \|B\| \| \widehat p\|   + \CS  \|u\| (|\widehat \theta| + |\widehat \theta_1|)   \Bigg)\\
& = \frac{\eps_M}{1-\eps_M \widetilde c_3 \kappa(B)} \Bigg( (1+ \CS  + \widetilde c_3 )  \kappa(B) \|x\|  +
\widetilde c_4 
 \kappa(B) \| \widehat p\| + \widetilde c_5  \kappa(B)\| (|\widehat \theta| + |\widehat \theta_1|) \Bigg)
\end{align*}
Thus, by defining $c_3$, $c_4$, and $c_5$ as in the statement of the theorem we obtain
\begin{equation}
\label{eq:MSM_FerrBnd2}
\| e \|
\le c_3 \eps_M \kappa(B)  \|x\| + c_4 \eps_M \kappa(B)  \|\widehat p\| + c_5 \eps_M \kappa(B) (|\widehat \theta| + |\widehat \theta_1|) 
\end{equation}
Then, $\|x\| \leq \|\widehat x\| + \|e\|$ yields~\eqref{eq:MSM_FerrBnd}.
\end{proof}

Using~\eqref{eq:MSM_FerrBnd2}, a relative forward error bound for MSM could be obtained:
\begin{equation}
 \label{eq:MSM-FerrBnd3}
\frac{\| \widehat x-x \|}{\|x\|}
 \le \eps_M \kappa(B) \Big( c_3 + \frac{c_4 \|\widehat p\| + c_5 (|\widehat \theta| + |\widehat \theta_1|)}{\|x\|}  \Big)
\end{equation}
which could then be made computable in a few different ways. The most straightforward one is using 
$\frac{1}{\|x\|}
    \leq
    \frac{\|B\|}{\|b\|}$ which yields
\begin{equation}
\frac{\| \widehat x-x \|}{\|x\|}
\le  \eps_M \kappa(B) \Big( c_3 + \frac{\|B\|}{\|b\|} \big( c_4 \|\widehat p\| + c_5 (|\widehat \theta| + |\widehat \theta_1|) \big)   \Big) \label{eq:MSM-computable-FerrBnd-0}
\end{equation}
We could instead use~\eqref{solution-norm-bound:eq} which gives
\begin{equation}
\frac{\| \widehat x-x \|}{\|x\|}
\le \eps_M \kappa(B) \Big( c_3 + \frac{c_4 \|\widehat p\| + c_5 (|\widehat \theta| + |\widehat \theta_1|)}{\|\widehat x\| - \|r\| \sigma_{\min}^{-1}(B)}  \Big) 
\label{eq:MSM-computable-FerrBnd}
\end{equation}

A third option would be to divide~\eqref{abs-for-err-msm-temp:eq} by $\|x\|$ obtaining
\begin{align}
\frac{\|\widehat x-x\|}{\|x\|}
&\le 
\frac{\eps_M}{1-\eps_M\widetilde c_3\kappa(B)}
\Bigg(
\frac{
(1+\CS)\|b\|
+\Big((2\CS+6)\|A\|+(2n+4)\|u\|\,\|v\|\Big)\|\widehat p\|
}{\|x\|\sigma_{\min}(B)} \nonumber
\\
&\qquad\qquad
+
\frac{
\CS\|u\|\big(|\widehat\theta|+|\widehat\theta_1|\big)
}{\|x\|\sigma_{\min}(B)}
+\widetilde c_3\kappa(B)
\Bigg) \nonumber \\
&\le 
\frac{\eps_M}{1-\eps_M\widetilde c_3\kappa(B)}
\Bigg(
\frac{
(1+\CS)\|b\|
+\Big((2\CS+6)\|A\|+(2n+4)\|u\|\,\|v\|\Big)\|\widehat p\|
}{\|\widehat x\| \sigma_{\min}(B) - \|r\|  }
\nonumber \\
&\qquad\qquad
+
\frac{
\CS\|u\|\big(|\widehat\theta|+|\widehat\theta_1|\big)
}{\|\widehat x\| \sigma_{\min}(B) - \|r\|  }
+\widetilde c_3\kappa(B)
\Bigg) \label{msm-ferr-bnd-option10:eq}
\end{align}
where the second inequality is based on~\eqref{solution-norm-bound:eq}.

The next result, which will be used in the sequel for obtaining a convenient test for MSM forward stability, follows from~\eqref{eq:MSM-FerrBnd3}. It simply defines $c_6$ such that it is larger than all three quantities $c_3, c_4$ and $c_5$.
\begin{corollary} \normalfont
With the notation of previous theorem, define
$ c_6 := \frac{1+\CS + \widetilde c_4}{1-\eps_M \widetilde c_3 \kappa(B)}$. The solution $\widehat x$ computed by MSM satisfies
\begin{equation}
\label{MSM-ferr-bnd-new:eq}
\frac{\| \widehat x-x \|}{\|x\|}
 \le \eps_M \kappa(B) c_6 \Big( 1 + \frac{\|\widehat p\| + |\widehat \theta| + |\widehat \theta_1|}{\|x\|}  \Big)
\end{equation}
\end{corollary}

\subsection{Discussion of error bounds and growth factors}
Here is a summary of what can be inferred from the error bounds established in previous sections. Our aim here is to present the results in a form that is a posteriori and easily computable. 

For both SM and MSM, we compare our backward and forward error bounds with the ideal error bounds that would guarantee backward or forward stability. We then assemble the extra terms in our bounds and introduce factors that explain how much our SM and MSM backward and forward error bounds degrade from the ideal bounds. We call them growth factors analogous to the pivot growth factors in Gaussian elimination which reveal when GE is stable. The SM and MSM backward and forward growth factors are all ratios that should ideally be $\mathcal{O}(1)$. We illustrate their use in the numerical experiments reported in the next section.

To derive the growth factors (for the forward error bounds) we drop terms in $c_j$ that depend on $n$ but extract terms that depend on $\|A\|$, $\|u\|$, $\|v\|$ and $\|B\| = \|A+uv^T\|$.
It turns out that even when the input problems are so ill-conditioned that some of the constants $c_j$ 
become negative, the growth factors may still correctly predict the stability behaviour of the methods. 

Recall that an algorithm for solving $Bx=b$ is backward stable, if its relative backward error is no larger --- up to a modest constant --- than machine epsilon, i.e., its computed solution $\widehat x$ satisfies $\omega(\hat{x}) = \mathcal{O} (\eps_M)$. Also, an algorithm for solving $Bx=b$ is said to be forward stable, if its relative forward error is no worse than what a backward stable method would produce, i.e., when the computed solution $\widehat x$ satisfies 
\[
\frac{\|\widehat x - x\|}{\|x\|} = \mathcal{O} (\eps_M \kappa(B)).
\]
Recall also that a forward stable method need not be backward stable and indeed, in the case of SM, we will see examples where it is forward stable, while not backward stable (see e.g., Figure~\ref{fig:ill_ill_smallNorm_tridiag}). In general, backward stability implies forward stability, but not vice versa. Another example of a method that is forward stable but not backward stable is Cramer's rule for solving a $2 \times 2$ linear system~\cite[p. 9]{Higham02}.

\begin{itemize}
\item According to~\eqref{newBackErrBnd1:eq}, SM is backward stable if the SM growth factor satisfies
\begin{equation}
\label{SMbackStabCriterion:eq}
\Gamma_{\rm{SM}}^{\rm{bw}} := \frac{\big( \|A\| + \|u\| \|v\| \big)\,\|\widehat y\|}{ \|A+u v^{T}\|\, \|\widehat x\| + \|b\|} = \mathcal{O} (1).
\end{equation}

\item We now discuss forward stability of SM. First, considering~\eqref{SM_relative_ferr0:eq} we can conclude that SM is forward stable when $\widecheck c_2 \|B\| \frac{\|\widehat y\|}{\|b\|}$ is $\mathcal{O} (1)$. Accounting for the non-trivial term $\frac{\|A\| + \|u\| \|v\|}{\|B\|}$ in ($\widetilde c_2$ and hence also in) $\widecheck c_2$ suggests checking if 
\begin{equation}
\label{SMforStabCriterion0:eq}
\Gamma_{\rm{SM}}^{\rm{fw}} := \frac{\big( \|A\| + \|u\| \|v\| \big) \|\widehat y\| }{\|b\|} = \mathcal{O} (1)
\end{equation}
holds. Similarly, with~\eqref{SM_relative_ferr:eq} one might look into the quantity
\[
\frac{\|A\| + \|u\| \|v\|}{\|B\|}  \frac{\|\widehat y\|}{\|\widehat x\| - \|r\| \sigma_{\min}^{-1}(A+uv^T)}  
\]
to conclude that if
\begin{equation}
\label{SMforStabCriterion:eq}
\Gamma_{\rm{SM}}^{\rm{fw}} := \frac{(\|A\| + \|u\| \|v\|) \|\widehat y\|}{\|A+uv^T\| \|\widehat x\| - \|r\| \kappa(A+uv^T)}  = \mathcal{O} (1)
\end{equation}
then SM is forward stable, where the denominator is assumed to be positive (otherwise the residual has not been small enough for the bound to carry useful information about $\|x\|$). Note that the denominator is a computable lower bound on $\|A+uv^T\| \|x\|$ and\footnote{obtained from a reverse-triangle estimate of $\|x\|$ in terms of $x-\widehat x=(A+uv^T)^{-1} r$ and multiplying through with $\|A+uv^T\|$.} essentially a substitute for $\|b\|$ once we account for the fact that $\widehat x \ne x$.

\item According to~\eqref{MSM_BackErrBnd:eq}, MSM is backward stable if the correction vector $\widehat p$ satisfies $\|\widehat p\| = \mathcal{O} (\|\widehat x\|)$ and the MSM growth factor satisfies
\begin{equation}
\label{MSMbackStabCriterion:eq}
\Gamma_{\rm{MSM}}^{\rm{bw}} :=  \frac{\|u\| \big( |\widehat \theta| + |\widehat \theta_1| \big) + \|b\|}{ \|A+u v^{T}\|\, \|\widehat x\| + \|b\|} = \mathcal{O} (1).
\end{equation}

\item We can combine~\eqref{MSM-ferr-bnd-new:eq} with bounds of the form $\frac{1}{\|x\|} \leq \frac{\|B\|}{\|b\|}$ or with~\eqref{solution-norm-bound:eq} to derive sufficient conditions for the forward stability of MSM that are convenient to check. We also make explicit the factor $\frac{\|A\| + \|u\| \|v\|}{\|A+uv^T\|}$ which is hidden in $c_6$. The conclusion is that MSM is forward stable if either
\begin{equation}
\label{MSMforStabCriterion1:eq}
\Gamma_{\rm{MSM}}^{\rm{fw}} := (\|A\| + \|u\| \|v\|) \frac{ \|\widehat p\| + |\widehat \theta| + |\widehat \theta_1|}{\|b\|} = \mathcal{O}(1)
\end{equation}
or
\begin{equation}
\label{MSMforStabCriterion:eq}
\Gamma_{\rm{MSM}}^{\rm{fw}} := (\|A\| + \|u\| \|v\|) \frac{ \|\widehat p\| + |\widehat \theta| + |\widehat \theta_1|}
{\|A+uv^T\| \|\widehat x\| - \|r\| \kappa(A+uv^T) } = \mathcal{O}(1)
\end{equation}
where in the last one, the denominator is assumed to be positive.
\end{itemize}

\section{Experiments}
We follow the setup introduced in~\cite{SM25}, where $A$ and $A+uv^T$ have varying condition numbers. To illustrate the performance of our new algorithm, MSM, we focus mainly on cases in which the solution has modest norm; in other cases, we observe that the standard SM is backward stable. As a reference, we report results obtained by applying MATLAB backslash to the matrix $B=A+uv^T$. In our setup, $B$ is unstructured (i.e. not sparse etc), and this approach is therefore always the slowest. To solve linear systems with $A$, we use LU decomposition with partial pivoting in SM, SMIR and MSM.

In the following plots, the horizontal axis reports the condition numbers of both $A$ and $B$. For instance, $10^{6, 12}$ means that $\kappa_2(A) \approx 10^6$ while $\kappa_2(B) \approx 10^{12}$.

The numerical results reported in the tables provide details of the backward and forward errors, together with the typical behaviour of our bounds. All these results use $\CS=1$. 
In the case of the backward error, the observed values are computed with~\eqref{eq:backErrB_def} while the SM and MSM backward error bounds refer to~\eqref{newBackErrBnd1:eq} and~\eqref{MSM_BackErrBnd:eq}, respectively. The SM forward error bound reported in the tables is the smallest among the three options~\eqref{SM_relative_ferr0:eq},~\eqref{SM_relative_ferr:eq} and~\eqref{sm-ferr-bnd-option3:eq}. Similarly, in the case of MSM, the forward error bound reported is the smallest obtained from among~\eqref{eq:MSM-computable-FerrBnd-0},~\eqref{eq:MSM-computable-FerrBnd} and~\eqref{msm-ferr-bnd-option10:eq}.

In the tables, in addition to the errors and their bounds, we report the numerator and denominator of
the growth factors $\Gamma_{\rm{SM}}^{\rm{bw}}$, $\Gamma_{\rm{MSM}}^{\rm{bw}}$, $\Gamma_{\rm{SM}}^{\rm{fw}}$ and $\Gamma_{\rm{MSM}}^{\rm{fw}}$ as defined in the last section. Note that a cross mark in the last column means only that the ratio obtained from the upper bound is not modest; it does not necessarily indicate a lack of (backward/forward) stability. By contrast, a tick mark means that the ratio is modest, and therefore implies stability. 

\begin{example} \label{ex:ill_ill_smallNorm_tridiag} \normalfont
We generate $n \times n$ random tridiagonal matrices $A$ with $n=4000$, choosing the parameters so that both $A$ and $B$ are ill-conditioned. The exact solution $x$, as well as the vectors $u$ and $v$, are taken to have normally distributed random entries generated with the MATLB command \texttt{randn}, leading to a solution of modest norm, and the right-hand side $b$ is then formed accordingly.

The results are depicted in Figure~\ref{fig:ill_ill_smallNorm_tridiag}. While SM is not backward stable, both MSM and SMIR are computing backward stable solutions while the latter requires up to six IR steps which slightly affects its performance. We also observe that all methods are forward stable.

\begin{figure}[!h]
\center
\includegraphics[width=0.95\textwidth]{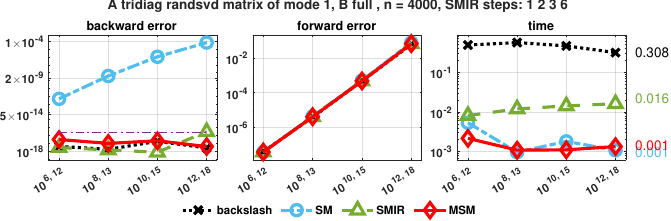}
\caption{Results for Example~\ref{ex:ill_ill_smallNorm_tridiag}: $\kappa(A), \kappa(A+uv^T) \gg 1$, modest-norm solution}
\label{fig:ill_ill_smallNorm_tridiag}
 \end{figure}

\begin{table}[!h]
\centering
\scriptsize
\caption{Backward errors, computed upper bounds, and their characteristic ingredients for the second problem in Example~\ref{ex:ill_ill_smallNorm_tridiag}. Here, $\kappa(A) \approx 10^8$ and $\kappa(A+uv^T) \approx 10^{13}$. See Fig.~\ref{fig:ill_ill_smallNorm_tridiag}.}
\label{tab:backward_errors-ill_ill_smallNorm_tridiag:tab}
\renewcommand{\arraystretch}{1.35}
\begin{tabular}{@{}lcc>{\raggedright\arraybackslash}p{0.30\textwidth}c@{}}
\toprule
Alg.
& Observed
& Bound
& Formula / check
& Value \\
\midrule
\multirow{3}{*}{SM}
& \multirow{3}{*}{$3.7 \times 10^{-9}$}
& \multirow{3}{*}{$1.6 \times 10^{-3}$}
& $\big( \|A\| + \|u\| \|v\| \big)\,\|\widehat y\|$
& $2.3 \times 10^{14}$ \\
&
&
& $\|A+u v^{T}\|\, \|\widehat x\| + \|b\|$
& $2.6 \times 10^{5}$ \\
&
&
& Is the ratio $\Gamma_{\rm{SM}}^{\rm{bw}} = \mathcal{O}(1)$?
& \xmark \\
\midrule
\multirow{3}{*}{MSM}
& \multirow{3}{*}{$9.8 \times 10^{-18}$}
& \multirow{3}{*}{$1.8 \times 10^{-12}$}
& $ \|u\| \big( |\widehat \theta| + |\widehat \theta_1| \big) + \|b\|$
& $1.2 \times 10^{3}$ \\
&
&
& $ \|A+u v^{T}\|\, \|\widehat x\| + \|b\|$
& $2.6 \times 10^{5}$ \\
&
&
& Is the ratio $\Gamma_{\rm{MSM}}^{\rm{bw}} = \mathcal{O}(1)$?
& \checkmark \\
\bottomrule
\end{tabular}
\end{table}

\begin{table}[!h]
\centering
\scriptsize
\caption{Forward errors, computed upper bounds, and their characteristic ingredients in the second problem of Example~\ref{ex:ill_ill_smallNorm_tridiag}. A dash indicates that the bound is invalid because one of the constants $c_j$ involved is negative. More precisely, required assumptions such as $\tilde c_3 \eps_M \kappa(A+uv^T) < 1$ do not hold because the test problem is simply too ill-conditioned, which is also why, for SM in this example, $\|A+uv^T\| \|\widehat x\| - \|r\| \kappa(A+uv^T)$ is negative. Therefore, in the last column, we report characteristic values in~\eqref{SMforStabCriterion0:eq} which correspond to the bound~\eqref{SM_relative_ferr0:eq}, rather than to~\eqref{SM_relative_ferr:eq}. See Fig.~\ref{fig:ill_ill_smallNorm_tridiag}.}
\label{tab:forward_errors-ill_ill_smallNorm_tridiag:tab}
\renewcommand{\arraystretch}{1.35}
\begin{tabular}{@{}lcc>{\raggedright\arraybackslash}p{0.45\textwidth}c@{}}
\toprule
Alg.
& Observed
& Bound
& Formula / check
& Value \\
\midrule
\multirow{3}{*}{SM}
& \multirow{3}{*}{$4.3 \times 10^{-6}$}
& \multirow{3}{*}{$-$}
& $ \|A+uv^T\| \|\widehat y\|$
& $2.3 \times 10^{14}$ \\
& & & $ \|b\|$ & $5.8 \times 10^{2}$ \\
& is forward stable
&
& Is the ratio $\Gamma_{\rm{SM}}^{\rm{fw}} = \mathcal{O}(1)$?
& \xmark \\
\midrule
\multirow{3}{*}{MSM}
& \multirow{3}{*}{$4.1 \times 10^{-6}$}
& \multirow{3}{*}{$-$}
& $(\|A\| + \|u\| \|v\|) \big(\|\widehat p\| + |\widehat \theta| + |\widehat \theta_1|\big)$
& $3.7 \times 10^4$ \\
& 
&
& $\|A+uv^T\| \|\widehat x\| - \|r\| \kappa(A+uv^T)$
& $2.6 \times 10^{5}$ \\
& is forward stable
&
& Is the ratio $\Gamma_{\rm{MSM}}^{\rm{fw}} = \mathcal{O}(1)$?
& \checkmark \\
\bottomrule
\end{tabular}
\end{table}

\end{example}

\newpage

\begin{example} \label{ex:ill_ill_largeNorm_tridiag} \normalfont
In this example, we first generate the right-hand side vector $b$ with entries drawn from a normal distribution. The main distinction from the previous example is that the solution now has a large norm. In this regime, SM already performs backward stably, and SMIR does not require any extra refinement steps. As Figure~\ref{fig:ill_ill_largeNorm_tridiag} illustrates, MSM is also backward stable. All methods are forward stable.
\begin{figure}[!h]
\center
\includegraphics[width=0.95\textwidth]{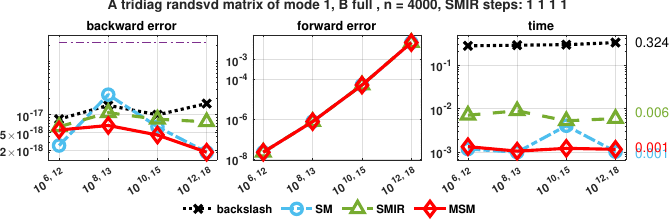}
\caption{Results for Example~\ref{ex:ill_ill_largeNorm_tridiag}: $\kappa(A), \kappa(A+uv^T) \gg 1$, large-norm solution}
\label{fig:ill_ill_largeNorm_tridiag}
 \end{figure}

 \begin{table}[!h]
\centering
\caption{Backward errors, computed upper bounds, and their characteristic ingredients in the second problem of Example~\ref{ex:ill_ill_largeNorm_tridiag}. Here, again $\kappa(A) \approx 10^8$ and $\kappa(A+uv^T) \approx 10^{13}$, but the solution is large-normed ($\|\widehat x\| \approx 10^9$). See Fig.~\ref{fig:ill_ill_largeNorm_tridiag}.}
\label{tab:backward_errors-ill_ill_largeNorm_tridiag:tab}
\scriptsize
\renewcommand{\arraystretch}{1.35}
\begin{tabular}{@{}lcc>{\raggedright\arraybackslash}p{0.30\textwidth}c@{}}
\toprule
Alg.
& Observed
& Bound
& Formula / check
& Value  \\
\midrule
\multirow{3}{*}{SM}
& \multirow{3}{*}{$2.4 \times 10^{-17}$}
& \multirow{3}{*}{$2.7 \times  10^{-12}$}
& $\big( \|A\| + \|u\| \|v\| \big)\,\|\widehat y\|$
& $2.5 \times 10^{13}$ \\
&
&
& $\|A+u v^{T}\|\, \|\widehat x\| + \|b\|$
& $2.6 \times 10^{13}$ \\
&
&
& Is the ratio $\Gamma_{\rm{SM}}^{\rm{bw}} = \mathcal{O}(1)$?
& \checkmark  \\
\midrule
\multirow{3}{*}{MSM}
& \multirow{3}{*}{$6.3 \times 10^{-18}$}
& \multirow{3}{*}{$1.8 \times 10^{-12}$}
& $ \|u\| \big( |\widehat \theta| + |\widehat \theta_1| \big) + \|b\|$
& $7.2 \times 10^1$ \\
&
&
& $\|A+u v^{T}\|\, \|\widehat x\| + \|b\|$
& $2.6 \times 10^{13}$ \\
&
&
& Is the ratio $\Gamma_{\rm{MSM}}^{\rm{bw}} = \mathcal{O}(1)$?
& \checkmark \\
\bottomrule
\end{tabular}
\end{table}

\begin{table}[!h]
\centering
\caption{Forward errors, computed upper bounds, and their characteristic ingredients in Example~\ref{ex:ill_ill_largeNorm_tridiag} in which solutions have a large norm. See Fig.~\ref{fig:ill_ill_largeNorm_tridiag}.}
\label{tab:forward_errors-ill_ill_largeNorm_tridiag:tab}
\scriptsize
\renewcommand{\arraystretch}{1.35}
\begin{tabular}{@{}lcc>{\raggedright\arraybackslash}p{0.40\textwidth}c@{}}
\toprule
Alg.
& Observed
& Bound
& Formula / check
& Value \\
\midrule
\multirow{3}{*}{SM} & \multirow{3}{*}{$7.6 \times 10^{-7}$} & \multirow{3}{*}{$-$} & $\big( \|A\| + \|u\| \|v\| \big) \|\widehat y\| $ & $2.5 \times 10^{13}$ \\
& 
&
& $\|A+uv^T\| \|\widehat x\| - \|r\| \kappa(A+uv^T)$
& $2.6 \times 10^{13}$ \\
& is forward stable
&
& Is the ratio $\Gamma_{\rm{SM}}^{\rm{fw}} = \mathcal{O}(1)$?
& \checkmark \\
\midrule
\multirow{3}{*}{MSM}
& \multirow{3}{*}{$7.6 \times 10^{-7}$}
& \multirow{3}{*}{$-$}
& $ (\|A\| + \|u\| \|v\|) \big(\|\widehat p\| + |\widehat \theta| + |\widehat \theta_1|\big)$
& $4.8 \times 10^{2}$ \\
&
&
& $\|A+uv^T\| \|\widehat x\| - \|r\| \kappa(A+uv^T)$
& $2.6 \times 10^{13}$ \\
& is forward stable
&
& Is the ratio $\Gamma_{\rm{MSM}}^{\rm{fw}} = \mathcal{O}(1)$?
& \checkmark \\
\bottomrule
\end{tabular}
\end{table}

\end{example}

\newpage

\begin{example} \label{ex:well_ill_smallNorm}  \normalfont
In the four test problems in this experiment, the condition number of $A$ takes the values ${10,10^2,10^3,10^4}$, while $B$ is more ill-conditioned. We set $n = 1000$. As Figure~\ref{fig:well_ill_smallNorm} shows, the backward stability of SM deteriorates as $A$ becomes more ill-conditioned, whereas both SMIR and MSM remain backward stable across all test problems, with MSM being the faster method.
\begin{figure}[!h]
\center
\includegraphics[width=0.95\textwidth]{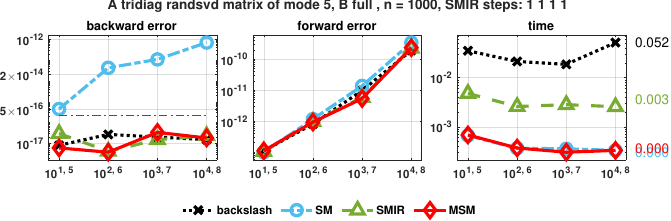}
\caption{Results for Example~\ref{ex:well_ill_smallNorm}: $\kappa(A)= \Theta(1), \kappa(A+uv^T)\gg 1$, modest-norm solution}
\label{fig:well_ill_smallNorm}
 \end{figure}

  \begin{table}[!h]
\centering
\caption{Backward errors, computed upper bounds, and their characteristic ingredients in the second problem of Example~\ref{ex:well_ill_smallNorm}. Here, $\kappa(A) \approx 10^2$ and $\kappa(A+uv^T) \approx 10^{6}$. Notice, in this example, SM is beginning to lose backward stability and although the backward error is not yet too large, the SM solution is actually not backward stable. See the blue curve in Fig.~\ref{fig:well_ill_smallNorm}.}
\label{tab:backward_errors-well_ill_smallNorm:tab}
\scriptsize
\renewcommand{\arraystretch}{1.35}
\begin{tabular}{@{}lcc>{\raggedright\arraybackslash}p{0.30\textwidth}c@{}}
\toprule
Alg.
& Observed
& Bound
& Formula / check
& Value \\
\midrule
\multirow{3}{*}{SM}
& \multirow{3}{*}{$4.3 \times 10^{-14}$}
& \multirow{3}{*}{$4.8 \times 10^{-10}$}
& $ \big( \|A\| + \|u\| \|v\| \big)\,\|\widehat y\|$
& $3.5 \times 10^{7}$ \\
&
&
& $\|A+u v^{T}\|\, \|\widehat x\| + \|b\|$
& $3.2 \times 10^{4}$ \\
&
&
& Is the ratio $\Gamma_{\rm{SM}}^{\rm{bw}} = \mathcal{O}(1)$?
& \xmark \\
\midrule
\multirow{3}{*}{MSM}
& \multirow{3}{*}{$3.8 \times 10^{-18}$}
& \multirow{3}{*}{$4.3 \times 10^{-13}$}
& $ \|u\| \big( |\widehat \theta| + |\widehat \theta_1| \big) + \|b\|$
& $1.9 \times 10^3$ \\
&
&
& $ \|A+u v^{T}\|\, \|\widehat x\| + \|b\|$
& $3.2 \times 10^4$ \\ &
& & Is the ratio $\Gamma_{\rm{MSM}}^{\rm{bw}} = \mathcal{O}(1)$?
& \checkmark \\
\bottomrule
\end{tabular}
\end{table}

\begin{table}[!h]
\centering
\caption{Forward errors, computed upper bounds, and their characteristic ingredients in the second problem of Example~\ref{ex:well_ill_smallNorm}. See Fig.~\ref{fig:well_ill_smallNorm}.}
\label{tab:forward_errors-well_ill_smallNorm:tab}
\scriptsize
\renewcommand{\arraystretch}{1.35}
\begin{tabular}{@{}lcc>{\raggedright\arraybackslash}p{0.40\textwidth}c@{}}
\toprule
Alg.
& Observed
& Bound
& Formula / check
& Value \\
\midrule
\multirow{3}{*}{SM}
& \multirow{3}{*}{$1.2 \times 10^{-12}$}
& \multirow{3}{*}{$7.2 \times 10^{-4}$}
& $\big( \|A\| + \|u\| \|v\| \big) \|\widehat y\| $
& $3.5 \times 10^7$ \\
&
&
& $\|A+uv^T\| \|\widehat x\| - \|r\| \kappa(A+uv^T)$
& $3.1 \times 10^4$ \\
& is forward stable
&
& Is the ratio $\Gamma_{\rm{SM}}^{\rm{fw}} = \mathcal{O}(1)$?
& \xmark \\
\midrule
\multirow{3}{*}{MSM}
& \multirow{3}{*}{$9.3 \times 10^{-13}$}
& \multirow{3}{*}{$3.3 \times 10^{-7}$}
& $ (\|A\| + \|u\| \|v\|) \big(\|\widehat p\| + |\widehat \theta| + |\widehat \theta_1|\big)$
& $3.0 \times 10^4$ \\
&
&
& $\|A+uv^T\| \|\widehat x\| - \|r\| \kappa(A+uv^T)$
& $3.1 \times 10^4$ \\
& is forward stable
&
& Is the ratio $\Gamma_{\rm{MSM}}^{\rm{fw}} = \mathcal{O}(1)$?
& \checkmark \\
\bottomrule
\end{tabular}
\end{table}

\end{example}

\newpage

\begin{example} \label{ex:ill_well_smallNorm}  \normalfont
 We now examine the case in which $B$ is well-conditioned, while $A$ is ill-conditioned and the solution has modest norm. As Figure~\ref{fig:ill_well_smallNorm} shows, the standard SM is neither backward stable, nor forward stable, whereas MSM (and SMIR) are both backward and forward stable.
\begin{figure}[!h]
\center
\includegraphics[width=0.95\textwidth]{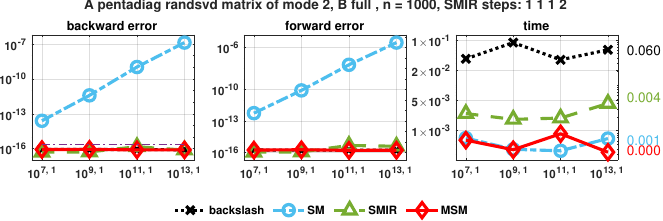}
\caption{Results for Example~\ref{ex:ill_well_smallNorm}: $\kappa(A)\gg 1, \kappa(A+uv^T) = \Theta(1)$, modest-norm solution}
\label{fig:ill_well_smallNorm}
 \end{figure}

\begin{table}[!h]
\centering
\caption{Backward errors, computed upper bounds, and their characteristic ingredients in the second problem of Example~\ref{ex:ill_well_smallNorm}. See normwise backward errors in Fig.~\ref{fig:ill_well_smallNorm}.}
\label{tab:backward_errors-ill_well_smallNorm:tab}
\scriptsize
\renewcommand{\arraystretch}{1.35}
\begin{tabular}{@{}lcc>{\raggedright\arraybackslash}p{0.30\textwidth}c@{}}
\toprule
Alg.
& Observed
& Bound
& Formula / check
& Value \\
\midrule
\multirow{3}{*}{SM}
& \multirow{3}{*}{$3.9\times 10^{-12}$}
& \multirow{3}{*}{$5.9 \times 10^{-7}$}
& $ \big( \|A\| + \|u\| \|v\| \big)\,\|\widehat y\|$
& $8.3 \times 10^7$ \\
&
&
& $ \|A+u v^{T}\|\, \|\widehat x\| + \|b\|$
& $6.3 \times 10^1$ \\
&
&
& Is the ratio $\Gamma_{\rm{SM}}^{\rm{bw}} = \mathcal{O}(1)$?
& \xmark \\
\midrule
\multirow{3}{*}{MSM}
& \multirow{3}{*}{$8.6 \times 10^{-17}$}
& \multirow{3}{*}{$2.4 \times 10^{-13}$}
& $ \|u\| \big( |\widehat \theta| + |\widehat \theta_1| \big) + \|b\|$
& $3.2 \times 10^{1}$ \\
&
&
& $ \|A+u v^{T}\|\, \|\widehat x\| + \|b\|$
& $6.3 \times 10^1$ \\
&
&
& Is the ratio $\Gamma_{\rm{MSM}}^{\rm{bw}} = \mathcal{O}(1)$?
& \checkmark \\
\bottomrule
\end{tabular}
\end{table}

 \begin{table}[!h]
\centering
\caption{Forward errors, computed upper bounds, and their characteristic ingredients in the second problem of Example~\ref{ex:ill_well_smallNorm}. Here, $\kappa(A) \approx 10^9$ while $\kappa(A+uv^T) \approx 10^{1}$. See forward errors in Fig.~\ref{fig:ill_well_smallNorm}.}
\label{tab:forward_errors-ill_well_smallNorm:tab}
\scriptsize
\renewcommand{\arraystretch}{1.35}
\begin{tabular}{@{}lcc>{\raggedright\arraybackslash}p{0.40\textwidth}c@{}}
\toprule
Alg.
& Observed
& Bound
& Formula / check
& Value \\
\midrule
\multirow{3}{*}{SM}
& \multirow{3}{*}{$8.1\times 10^{-11}$}
& \multirow{3}{*}{$1.1 \times 10^{-6}$}
& $\big( \|A\| + \|u\| \|v\| \big) \|\widehat y\| $
& $8.3 \times 10^7$ \\
&
&
& $\|A+uv^T\| \|\widehat x\| - \|r\| \kappa(A+uv^T)$
& $3.2 \times 10^1$ \\
& not forward stable
&
& Is the ratio $\Gamma_{\rm{SM}}^{\rm{fw}} = \mathcal{O}(1)$?
& \xmark \\
\midrule
\multirow{3}{*}{MSM}
& \multirow{3}{*}{$1.8 \times 10^{-16}$}
& \multirow{3}{*}{$2.7\times 10^{-12}$}
& $ (\|A\| + \|u\| \|v\|) \big(\|\widehat p\| + |\widehat \theta| + |\widehat \theta_1|\big)$
& $4.6 \times 10^{-1}$ \\
&
&
& $\|A+uv^T\| \|\widehat x\| - \|r\| \kappa(A+uv^T)$
& $3.2 \times 10^{+1}$ \\
& is forward stable
&
& Is the ratio $\Gamma_{\rm{MSM}}^{\rm{fw}} = \mathcal{O}(1)$?
& \checkmark \\
\bottomrule
\end{tabular}
\end{table}

\end{example}

\subsection{Does small capacitance cause backward instability in SM and MSM?}\label{smallBeta:sec}
As noted in Section 1, the focus of Yip's work~\cite{Yip86} is on the {\em forward} stability of SM. Yip examines situations in which the capacitance $\beta$ is close to zero (or, in the rank-$k$ perturbation case with $k>1$, nearly singular). A natural question is whether a small capacitance also degrades the {\em backward} error, particularly in light of the condition~\eqref{hypoth1:eq} required by the SM backward error bound established in our earlier work~\cite{SM25}. 

Two features of our analysis suggest that the backward error may tell a different story. First, while we still require $\beta \neq 0$ (otherwise $B$ is singular), the bounds in Section 2 are derived without an assumption like~\eqref{hypoth1:eq}. Second, while one may imagine small capacitance causes the quantity $\theta = \alpha / \beta$ to blow up, it does not appear in the new backward error representation. One is therefore led to expect a marked contrast with the forward error behaviour.

The MSM picture is, at first sight, more delicate. Unlike the SM bound, our MSM bound does involve the capacitance: it appears in the denominators of both $\widehat \theta$ and $\widehat \theta_1$. This raises a natural concern. If one deliberately chooses $u$ or $v$ so that $\beta = 1+v^TA^{-1}u\approx 0$, does the MSM backward error suffer? We test this in slightly different ways in the following two examples.

\begin{example} \label{sing_capacitance:ex}  \normalfont
We fix a $500 \times 500$ matrix $A$ with normally distributed random entries and $\kappa(A) = 10^5$, together with random vectors $v$, $x$ and $z_0$. The vector $x$ is taken to be the exact solution and has modest norm, here $\|x\| \approx 21.8$. We then generate 100 test problems with $10^{-16} \leq \delta_j \leq 10^{-1}$ for $j=1, 2, \dots, 100$. For each $j$ we define $z:= \frac{-1+\delta_j}{v^T z_0} z_0$ and then set $u:= A z$ and for the right-hand side vector $b := (A+uv^T) x$. With this construction, the capacitance in the $j$-th problem instance is equal to $\delta_j$ and these values are shown on the horizontal axis in Figure~\ref{fig:SM_MSM_den_zero_backErr}. The observed backward errors for the solutions computed by SM, MSM and MATLAB backslash applied to $B$ are then depicted together with the SM and MSM backward error bounds~\eqref{newBackErrBnd1:eq} and~\eqref{MSM_BackErrBnd:eq}, respectively.
\begin{figure}[!h]
\center
\includegraphics[width=0.7\textwidth]{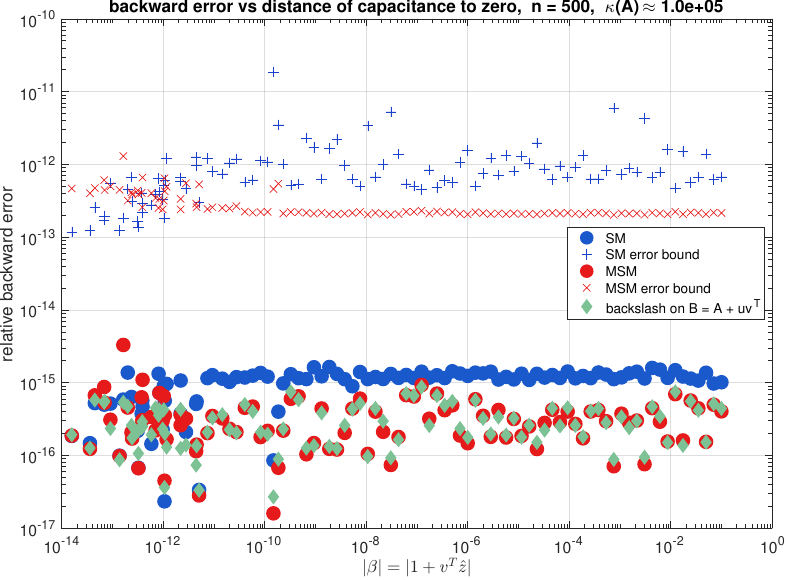}
\caption{Backward error of different methods for  Example~\ref{sing_capacitance:ex}.}
\label{fig:SM_MSM_den_zero_backErr}
 \end{figure}

We observe that, even in this artificial construction, the MSM backward error does not grow --- the largest such value among all 100 problems is $3.3 \times 10^{-15}$. Similarly, the largest SM and backslash backward errors are $1.6 \times 10^{-15}$ and $8.7 \times 10^{-16}$, respectively.

Although $\widehat \beta$ is tiny, the numerator $\widehat \alpha$ shrinks in proportion, leaving $\widehat \theta$ at order one. A concrete run illustrates this: $\widehat \beta \approx 3.3 \times 10^{-13}$ together with $\widehat \alpha \approx 7.8 \times 10^{-12}$ produces $\widehat \theta \approx 23.8$. The resulting SM solution has only a modest norm; it is not accurate in absolute terms --- the residual norm is about $10^{-8}$ --- yet, with $B$ and $b$ both of order $10^6$, the absolute error of $10^{-8}$ is absorbed in the denominator of the SM relative backward error and the latter remains acceptable. For reference, the same run gives $\|\widehat y\| \approx 52$ and $\|\widehat z\| \approx 0.9$. Regarding MSM on the same instance, the correction numerator $v^T w - \theta$ comes out at roughly $2.3 \times 10^{-11}$, which is no surprise since this quantity is zero in exact arithmetic. The aforementioned small magnitude is exactly what is needed to balance the equally small capacitance $\widehat \beta$, with the result that $\widehat \theta_1\approx 69.8$ --- once again of order one.

Even though cancellation occurs in the addition $1+v^T\widehat z$ with a tiny $\widehat \beta$, in the case of SM, we do not slip into the instability regime described in~\cite{SM25}, where the intermediate vectors $\widehat y$ and $\widehat z$ grow large while the final solution $\widehat x$ stays of modest norm. In short, a small capacitance is not, in itself, a mechanism for inflating $\| \widehat y\|$ and $\|\widehat z\|$ while keeping $\|\widehat x\|$ modest, and the SM solution, too, remains backward stable.

It is worth emphasising that a small capacitance is essentially a statement about the problem rather than about the algorithm: it signals that $B = A + uv^T$ is ill-conditioned. Forward error growth in this regime is therefore unavoidable and would be observed with any solver, including, for instance, MATLAB's backslash applied to $B$. In this experiment, the condition number of $B$ varies between $10^6$ and $10^{19}$.
\end{example}

\begin{example}\label{sing_capacitance2:ex}  \normalfont
This time we take a fixed \texttt{randsvd} matrix of the same size as before with $\kappa(A) \approx 10^{10}$, together with fixed random vectors $u$, $b$ and $v_0$. For the same values of $\delta_j$ as in the previous example, we set $v := \frac{-1 + \delta_j}{v_0^T A^{-1} u}\, v_0$, so that $\beta_j = 1 + v^T A^{-1} u \approx \delta_j$ as before. The results are shown in Figure~\ref{fig:SM_MSM_den_zero_backErr2}.
\begin{figure}[!h]
\center
\includegraphics[width=0.8\textwidth]{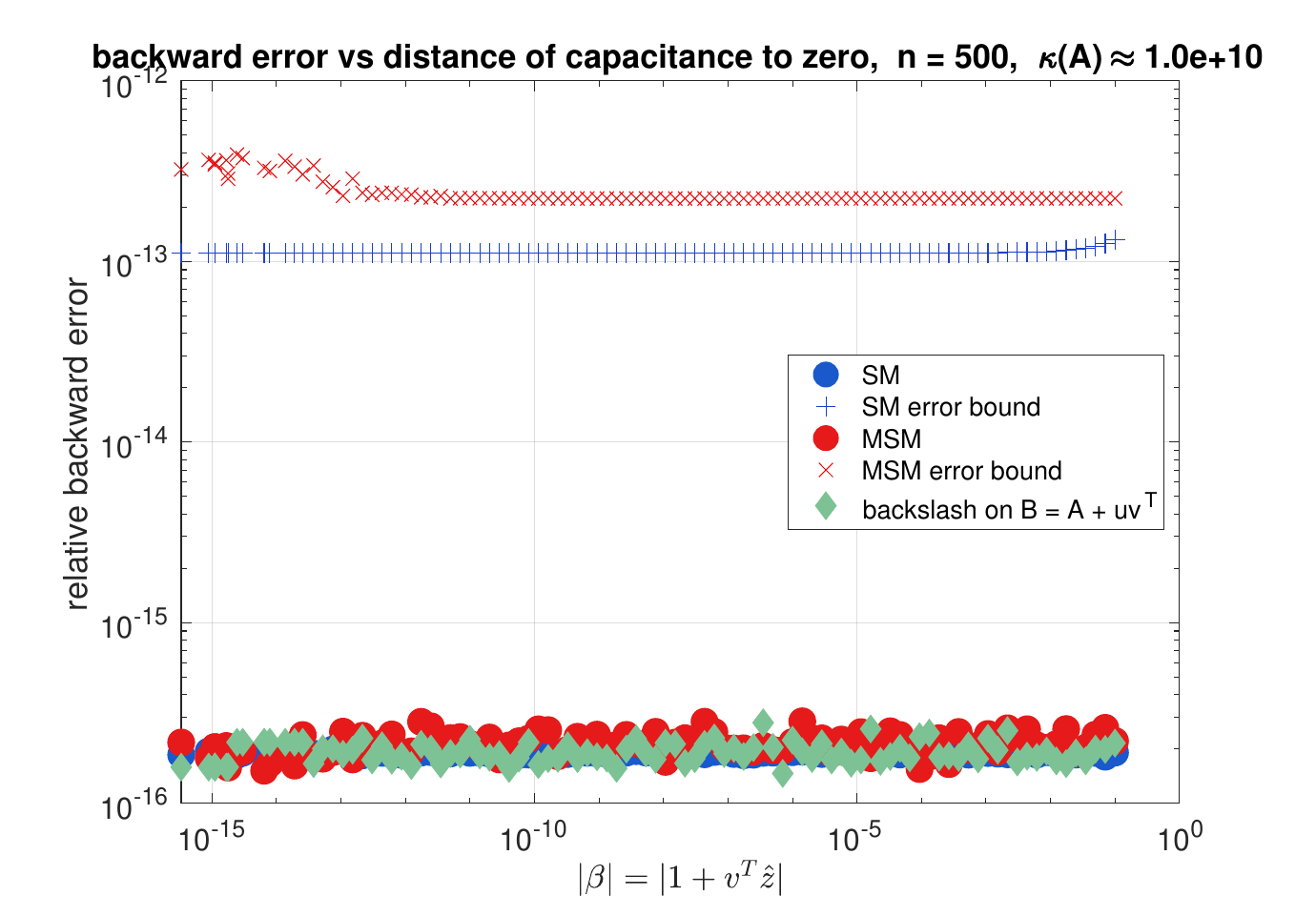}
\caption{Backward error of different methods for the Example~\ref{sing_capacitance2:ex}.}
\label{fig:SM_MSM_den_zero_backErr2}
 \end{figure}
 
 Again there is no loss of backward stability in either SM or MSM, but the mechanism is different from the previous example. This time the computed solutions $\widehat x$ themselves have very large norms, and that growth enters the denominator of the Rigal--Gaches formula~\eqref{eq:backErrB_def}, absorbing the inflated residuals.

To be concrete, in one problem instance we find $\widehat\beta \approx 3.0 \times 10^{-13}$ paired with $\widehat\alpha \approx 1.6$: in contrast to the previous example, $\widehat\alpha$ does not shrink in proportion with $\widehat\beta$. The quotient $\widehat\theta = \widehat\alpha/\widehat\beta \approx 4.3 \times 10^{12}$ is correspondingly huge, and all three methods --- SM, MSM, and backslash on $B$ --- return solutions of comparably large norm, with backslash producing the slightly smallest. The absolute residual norms of SM and MSM are both of order $10^{7}$, large by any standalone measure; yet with $\|\widehat x\| \sim 10^{23}$, the product $\|B\|\,\|\widehat x\|$ in the Rigal--Gaches denominator controls these residuals and the relative backward error remains small. For reference, the same run gives $\|\widehat y\| \approx \|\widehat z\| \approx 10^{10}$. The largeness of $\widehat\theta$ has a second consequence: $\widehat\theta\,\widehat z$ dominates $\widehat y$ in $\widehat x = \widehat y - \widehat\theta\,\widehat z$, so no catastrophic cancellation arises when SM assembles $\widehat x$. For MSM on the same instance, the correction numerator $v^T w - \theta$ evaluates to roughly $7.9 \times 10^{-2}$, yielding $\widehat\theta_1 \approx 10^{11}$. What keeps the backward error small here is therefore not that the algorithm sidesteps small $\widehat\beta$ --- it does not --- but that $\|\widehat x\|$ inflates by precisely the factor needed for the Rigal--Gaches denominator to absorb the inflated residual.

It is worth pausing to ask whether such solutions are meaningful in applications. In practice, a large-norm solution might be a mathematical artefact rather than a usable answer; see~\cite[pp. 101-102]{gill2021numerical} for instance. In our context, driving $\widehat\beta \to 0$ steers $B = A + uv^T$ towards singularity by construction, so the underlying linear system loses physical content long before $\widehat x$ reaches such magnitudes. Encountering $\widehat\beta \ll 1$ with $\widehat\alpha$ of order one is therefore best read as a diagnostic --- the perturbation has rendered the system close to singular.
\end{example}

\section*{Summary and future directions}
In this paper we have introduced the Modified Sherman-Morrison (MSM) algorithm and derived backward and forward error bounds for both SM and MSM. These bounds yield simple growth factors that are easy to compute a posteriori and, when of order one, act as sufficient conditions for the backward and forward stability of both algorithms. 

We have been unable to construct adversarial examples in which MSM is either backward or forward unstable, despite targeted attempts to do so. An unconditional proof --- or disproof --- of the backward or forward stability of MSM remains elusive, and we leave both as open questions for future work; the bounds developed here provide a natural foundation on which such an analysis could be built.

It would also be interesting to extend MSM to the general rank-$k$ Sherman-Morrison-Woodbury (SMW) formula, and to least-squares problems, exploring the extent to which the self-correction strategy of MSM can enhance the stability properties of SMW and \texttt{WoodburyLS}~\cite{guttel2024sherman}.

\section*{Acknowledgments}
We are grateful to David Bindel for pointing out the connection between the Sherman-Morrison formula and block elimination for bordered linear systems, and for bringing~\cite{Govaerts91} to our attention.  

\bibliography{MSM} 

\appendix
\bigskip
\section{Proof of an elementary result}\label{sec:App1}

\begin{lemma}
\label{lem:C_B_prime}
If assumptions of Theorem~\ref{thm:SM-back-bnd} are satisfied and 
\[
\widetilde c_1 := (n+2+\CS) \frac{\|A\| + \|u\| \|v\|}{\|B\|}
\]
then $\|\Delta B\| \le \eps_M \widetilde c_3 \|B\|$, where second-order terms in $\eps_M$ are dropped.
\end{lemma}
\begin{proof}
Taking norms from both sides of
\[
\Delta B:= \Delta A + u\ \Delta v^T+\Delta u\ v^T +\Delta u\ \Delta v^T
\]
gives
\begin{align*}
\| \Delta B\| & \leq \| \Delta A\| + \|u\| \ \|\Delta v\|+ \| \Delta u\| \ \|v\| +\| \Delta u\| \ \|\Delta v\|\\
& \leq (2+\CS) \eps_M \| A\| + \|u\| (n+2) \eps_M \| v\|+ \CS \eps_M \| u\| \ \|v\| + \mathcal O(\eps_M^2)\\
& = \eps_M \Big( (2+\CS) \| A\| + (n+2+\CS) \|u\| \| v\| \Big)+ \mathcal O(\eps_M^2)\\
& \leq \eps_M (n+2+\CS) \Big( \| A\| +  \|u\| \| v\| \Big)+ \mathcal O(\eps_M^2)\\
& = \eps_M (n+2+\CS) \Big(\frac{ \| A\| +  \|u\| \| v\|}{\|B\|} \Big) \|B\|+ \mathcal O(\eps_M^2)
\end{align*}
\end{proof}

\end{document}